\documentclass{amsart}
\usepackage{hyperref}
\usepackage{stmaryrd}
\usepackage[all]{xy}
\RequirePackage{amsmath}
\RequirePackage{amssymb}
\RequirePackage{amsxtra}
\RequirePackage{amsfonts}
\RequirePackage{latexsym}
\RequirePackage{euscript}
\RequirePackage{amscd}
\RequirePackage{amsthm}

\newtheorem{prop}{Proposition}[section]

\newtheorem{lem}[prop]{Lemma}
\newtheorem{thm}[prop]{Theorem}

\newtheorem{remar}[prop]{Remark}

\newtheorem{cor}[prop]{Corollary}

\DeclareMathAlphabet{\mathpzc}{OT1}{pzc}{m}{it}

\DeclareMathOperator{\End}{End}
\DeclareMathOperator{\Hom}{Hom}

\DeclareMathOperator{\Ind}{Ind}
\DeclareMathOperator{\cInd}{c-Ind}
\DeclareMathOperator{\Res}{Res}
\DeclareMathOperator{\Sym}{Sym}

\DeclareMathOperator{\GL}{GL}
\DeclareMathOperator{\SL}{SL}

\DeclareMathOperator{\WD}{WD}
\DeclareMathOperator{\Gal}{Gal}

\DeclareMathOperator{\soc}{soc}

\DeclareMathOperator{\tr}{tr}

\newcommand{\gr}{\mathrm{gr}}
\DeclareMathOperator{\Spec}{Spec}
\DeclareMathOperator{\mSpec}{m-Spec}

\DeclareMathOperator{\Mod}{Mod}

\DeclareMathOperator{\val}{val}

\DeclareMathOperator{\Tor}{Tor}
\DeclareMathOperator{\Ext}{Ext}

\DeclareMathOperator{\dt}{det}
\DeclareMathOperator{\Ban}{Ban}

\DeclareMathOperator{\dualcat}{\mathfrak C}

\newcommand{\Q}{\mathbb{Q}}
\newcommand{\Qp}{\mathbb {Q}_p}
\newcommand{\Zp}{\mathbb{Z}_p}
\newcommand{\Qpbar}{\overline{\mathbb{Q}}_p}

\newcommand{\PP}{\mathbb P}

\newcommand{\FF}{\mathcal F}

\newcommand{\ZZ}{\mathbb Z}
\newcommand{\VV}{\mathbf V}

\newcommand{\Sc}{\mathcal S}

\newcommand{\QQ}{\mathbb Q}

\newcommand{\Fbar}{\overline{F}}

\newcommand{\mm}{\mathfrak m}

\newcommand{\pr}{\mathrm{pr}}

\newcommand{\OO}{\mathcal O}

\newcommand{\TT}{\mathbb T}

\DeclareMathOperator{\wtimes}{\widehat{\otimes}}

\newcommand{\cV}{\check{\mathbf{V}}}

\newcommand{\BB}{\mathfrak B}

\newcommand{\md}{\mathrm m}

\newcommand{\pp}{\mathfrak p}

\newcommand{\sm}{\mathrm{sm}}

\newcommand{\pro}{\mathrm{pro}}

\newcommand{\adm}{\mathrm{adm}}
\newcommand{\alg}{\mathrm{alg}}
\newcommand{\cont}{\mathrm{cont}}

\newcommand{\ladm}{\mathrm{l.adm}}

\newcommand{\glt}{\mathfrak{gl}_2}

\newcommand{\rec}{\mathrm{rec}}

\newcommand{\T}{\mathbb{T}}
\def\et{\mathrm{\acute{e}t}}
\newcommand{\slt}{\mathfrak{sl}_2}

\DeclareMathOperator{\Frob}{Frob}

\newcommand{\Fbreve}{\breve{F}}
\newcommand{\Cp}{\mathbb C_p}
\newcommand{\rbar}{\bar{r}}
\newcommand{\rbarss}{\bar{r}^{\mathrm{ss}}}
\newcommand{\sigmabar}{\bar{\sigma}}
\newcommand{\rhobar}{\bar{\rho}}
\newcommand{\ps}{\mathrm{ps}}
\newcommand{\univ}{\mathrm{univ}}
\newcommand{\lalg}{\mathrm{l.alg}}
\newcommand{\Shat}{\check{\mathcal S}}
\newcommand{\AfF}{\mathbb{A}_F^f}
\newcommand{\cyc}{\mathrm{cyc}}
\newcommand{\Nrd}{\mathrm{Nrd}}
\newcommand{\wt}{\mathbf{w}}
\newcommand{\onealg}{1\text{-}\mathrm{alg}}

\title{On some consequences of a theorem of J. Ludwig}
\author{Vytautas Pa\v{s}k\={u}nas}
\date{\today.}
\begin{document} 
\maketitle

\begin{abstract} We prove some qualitative results about the  $p$-adic Jacquet--Lang\-lands
correspondence defined by Scholze, in the $\GL_2(\Qp)$, residually reducible case, by 
using a vanishing theorem proved by Judith Ludwig. In particular, we show that in the cases 
under consideration the global $p$-adic Jacquet--Langlands correspondence can also 
deal with automorphic forms with principal series representations at $p$ in a non-trivial way, unlike its classical counterpart. 
\end{abstract}

\section{Introduction} 

Let $F$ be a finite extension of $\Qp$, and let $L$ be a further sufficiently large finite extension of 
$F$, which will serve as the field of coefficients. Let $\OO$ be the ring of integers in $L$, $\varpi$ a uniformizer in $\OO$ and let $\OO/\varpi=k$ be the residue field. 
 
 To  a smooth admissible representation $\pi$ of $\GL_n(F)$ on an $\OO$-torsion module Scholze 
 in \cite{scholze} attaches a sheaf $\FF_{\pi}$ on the adic space $\PP^{n-1}_{\Cp}$ and shows that the cohomology 
 groups $H^i_{\et}(\PP^{n-1}_{\Cp}, \FF_{\pi})$ are admissible representations of $D_p^{\times}$, the group of units in a central division over $F$ with invariant $1/n$, and carry a continuous commuting action of $G_F$, the absolute Galois group of $F$. His construction is expected to realize both $p$-adic local Langlands and $p$-adic Jacquet--Langlands correspondences. However, these groups seem to 
 be very hard to compute, even to decide whet  her they are zero or not is highly non-trivial.
 
 In order to extend his construction to admissible unitary Banach space representations $\Pi$ of $\GL_n(F)$, the following seems like a sensible thing to do: choose an open bounded $\GL_n(F)$-invariant lattice $\Theta$ in $\Pi$, then $\Theta/\varpi^m$ is an admissible smooth representation
 of $\GL_n(F)$ on an $\OO$-torsion module and we may consider the limit $\varprojlim_m H^i_{\et}(\PP^{n-1}_{\Cp}, \FF_{\Theta/\varpi^m})$ equipped with the $p$-adic topology. We would like to invert $p$ and obtain a Banach space, but the $\OO$-module 
 might not be $\OO$-torsion free, and once we quotient out by the torsion, it might not be Hausdorff.
 Hence, it seems sensible to define $\Shat^i(\Pi):=Q(\Theta)\otimes_{\OO} L$, where $Q(\Theta)$ is 
 the maximal Hausdorff, $\OO$-torsion free quotient of   $\varprojlim_m H^i_{\et}(\PP^{n-1}_{\Cp}, \FF_{\Theta/\varpi^m})$. Even if we could compute  $H^i_{\et}(\PP^{n-1}_{\Cp}, \FF_{\Theta/\varpi})$
 and show it is non-zero we cannot conclude that $\Shat^i(\Pi)$ is non-zero. However, it is easy to see that   if the neighbouring groups $H^{i-1}_{\et}(\PP^{n-1}_{\Cp}, \FF_{\Theta/\varpi})$ and
 $H^{i+1}_{\et}(\PP^{n-1}_{\Cp}, \FF_{\Theta/\varpi})$ both  vanish then the non-vanishing of $H^i_{\et}(\PP^{n-1}_{\Cp}, \FF_{\Theta/\varpi})$ implies the non-vanishing of $\Shat^i(\Pi)$, see section \ref{dualcat}.

 If $F=\Qp$ and $n=2$ and $\pi$ is a principal series representation of $\GL_2(\Qp)$ then 
 Judith Ludwig shows in \cite{ludwig} the vanishing of $H^i_{\et}(\PP^1_{\Cp}, \FF_{\pi})$, when $i=2$. All the other groups for $i>2$ are known to vanish due to Scholze in this case. 
 Moreover, it follows from his results that $H^0(\PP^1_{\Cp}, \FF_{\pi})$ also vanishes, if $\pi^{\SL_2(\Qp)}$ is equal to zero. Thus if we restrict our attention only to those representations $\pi$, which have all irreducible subquotients isomorphic to
 irreducible principal series representations, then the functor $\pi\mapsto H^1_{\et}(\PP^1_{\Cp}, \FF_{\pi})$ is exact. Such subcategories of smooth representations 
 of $\GL_2(\Qp)$ on $\OO$-torsion modules have been studied in \cite{image} and shown to be related to the 
 reducible  $2$-dimensional mod $p$ representations of $G_{\Qp}$. The exactness of the functor 
 $\pi\mapsto H^1_{\et}(\PP^1_{\Cp}, \FF_{\pi})$ allows us to use arguments of Kisin in \cite{kisin}, who used the
 exactness of Colmez's functor to make a connection between the deformation theory of 
 $\GL_2(\Qp)$-representations and deformation theory of $2$-dimensional $G_{\Qp}$-representations. For simplicity we assume that $p>2$ in this paper. 
 
  \begin{thm}\label{thm_A}
  Let $r: G_{\Qp}\rightarrow \GL_2(L)$ be a continuous representation 
with $\rbarss= \chi_1\oplus \chi_2$, where 
$\chi_1, \chi_2: G_{\Qp}\rightarrow k^{\times}$ are characters, such that $\chi_1\chi_2^{-1}\neq \omega^{\pm 1}$, where $\omega$ is the  mod $p$ cyclotomic character.
Let $\Pi$ be the admissible unitary $L$-Banach space representation of $\GL_2(\Qp)$ 
corresponding to $r$ via the $p$-adic local Langlands correspondence for $\GL_2(\Qp)$. Then 
$\Shat^1(\Pi)\neq 0$. 
\end{thm}
 
 Previously such a result was known only in the case, when $r$ is a (twist of a) potentially semi-stable, 
 non-crystabelline representation, which lies on an automorphic component of a potentially semistable deformation ring, proved by Chojecki and Knight in \cite{chojecki_knight}. They prove it by patching and showing that locally algebraic vectors in $\Shat^1(\Pi)$ are non-zero. Their  argument relies on the theorem of Emerton \cite{interpolate}  which allows to interpret classical automorphic forms as locally algebraic vectors in completed cohomology, and enables them to handle only the representations which are "discrete series at $p$". For example, their argument 
 does not work for representations which become crystalline after restriction to a Galois group 
 of an abelian extension of $\Qp$ (prinicipal series at $p$), or representations, which do not become 
 potentially semistable after twisting by a character (non-classical). 
 
 Our argument works as follows: by the mod $p$ local Langlands correspondence to $\rbarss$ one may associate two principal series representations $\pi_1$ and $\pi_2$. We know from \cite{CDP}, 
 \cite{image} that the semi-simplification of $\Pi^0/\varpi$ is isomorphic to $\pi_1\oplus\pi_2$, 
 where $\Pi^0$ is a unit ball in $\Pi$.
Using the exactness results described above it is enough to show that at least one of
$H^1_{\et}(\PP^1_{\Cp}, \FF_{\pi_1})$, $H^1_{\et}(\PP^1_{\Cp}, \FF_{\pi_2})$ does not vanish. 
To show that, it is enough to find some $\pi$ with all irreducible subquotients isomorphic to 
either $\pi_1$ or $\pi_2$, such that $H^1_{\et}(\PP^1_{\Cp}, \FF_{\pi})$ does not vanish. As we have 
already mentioned, it seems  impossible to compute 
$H^1_{\et}(\PP^1_{\Cp}, \FF_{\pi})$ in general. However, Scholze manages to do so for certain representations coming from geometry. If $\pi=S(U^p, \OO/\varpi^n)$ is a  $0$-th completed cohomology group of a tower of zero dimensional 
Shimura varieties associated to a quaternion algebra $D_0$ over\footnote{In the main body of the paper we work with a totally real field $F$, such that $F_{\pp}=\Qp$ for a place $\pp$. We assume that $F=\QQ$ for the purpose of this introduction.} $\QQ$ 
 which is split at  $p$ and ramified at $\infty$, Scholze  shows in \cite{scholze} that 
 $H^1_\et(\PP^1_{\Cp}, \FF_\pi)$ is isomorphic as $G_{\Qp}\times D_p^\times$-representation to the $1$-st completed cohomology group $\widehat{H}^1(U^p, \OO/\varpi^n)$ of a  tower of Shimura curves associated to a quaternion algebra $D$, which is 
 ramified at $p$, split at $\infty$, and has the same ramification as $D_0$ at all the other places. 
 Here $U^p$ denotes some fixed tame level. 
 Scholze  also shows that this isomorphism respects the action by Hecke operators on both sides. We show that a localization of $S(U^p, \OO/\varpi^n)$ at a maximal 
 ideal $\mm$ of the Hecke algebra corresponding to an absolutely irreducible Galois representation $\rhobar: G_{\QQ, S}\rightarrow \GL_2(k)$, such that the semi-simplification of $\rhobar|_{G_{\Qp}}$ is equal $\chi_1\oplus \chi_2$ as above, has all irreducible subquotients isomorphic to $\pi_1$ or $\pi_2$ as above. By applying Scholze's functor to this representation we obtain $\widehat{H}^1(U^p, \OO/\varpi^n)_\mm$, which we can show is non-zero by using the classical Jacquet--Langlands correspondence. 
 
 Let $K$ be any open uniform pro-$p$ subgroup of $D_p^{\times}$, then its completed group
 ring $\OO[[K]]$ and the localization $\OO[[K]]\otimes_\OO L$ are both Auslander regular rings, 
 and as a consequence there is a good dimension theory for their finitely generated modules, 
 generalizing the Krull dimension.  If $\mathrm B$ is an admissible unitary Banach space 
 representation of $D_{p}^{\times}$ then its Schikhof dual $\mathrm B^d$ is a finitely generated
 $\OO[[K]]\otimes_\OO L$-module and we define the $\delta$-dimension of $\mathrm B$ to be the 
 dimension of $\mathrm B^d$ in the above sense.
 
 \begin{thm}\label{thm_B} If  $\Pi$ is as in Theorem \ref{thm_A} then the $\delta$-dimension 
 of $\Shat^1(\Pi)$ is one. 
 \end{thm}
 
 The Theorem is proved by using the observation of Kisin in \cite{kisin} that exact functors take flat 
 modules to flat modules and the results of Gee--Newton \cite{gee_newton} on miracle flatness in non-commutative setting. 
 
 One can show that a Banach space representation of $D_p^{\times}$ has $\delta$-dimension  
 zero if and only if it is finite dimensional as an $L$-vector space. Moreover, the zero dimensional 
 representations build a Serre subcategory, and thus one may pass to a quotient category. Informally this means that two $1$-dimensional Banach space representations are isomorphic in the quotient category if they differ by a $0$-dimensional Banach space representation. As Kohlhaase has pointed out to me, the  quotient category
 of Banach space representation of $\delta$-dimension at most $1$ by the zero dimensional Banach space representations  is noetherian and has an involution, hence it is also
 artinian, and hence every object in the quotient category has finite length. From this it is easy to deduce the following. 
  
  \begin{cor}\label{Cor_A} If  $\Pi$ is  as in Theorem \ref{thm_A} then $\Shat^1(\Pi)$ is of finite length 
  in the category of admissible unitary $L$-Banach space representations of $D_{p}^{\times}$ if and only if it has finitely many irreducible subquotients, which are finite dimensional as $L$-vector 
  spaces. 
 \end{cor} 
 
 We got quite excited about this Corollary at first, since we hoped that it might imply that $\Shat^1(\Pi)$ is of finite length as a Banach space representation of $D_p^\times$ by some formal representation theoretic arguments. However, here is an example suggesting that one should be cautious. If $K=\Zp$ then $\OO[[K]]\cong \OO[[x]]$, a commutative formal power series ring in one variable and  the Banach space of continuous functions on $K$ has $\delta$-dimension 
 one and is irreducible in the quotient category, which is equivalent to the category of finite dimensional vector spaces over the fraction field of $\OO[[x]]$; its Schikhof dual is isomorphic to 
 $\OO[[x]] [1/p]$ 
 and all irreducible subquotients are finite dimensional. See however, Theorem \ref{thm_D} below.
 
 \subsection{Global $p$-adic Jacquet--Langlands correspondence} Theorem \ref{thm_A} has a global application, which we state in the introduction for $F=\QQ$. The 
results are proved for a  totally real field $F$ such that $F_{\pp}=\Qp$, see Theorems \ref{main}, \ref{main_2} and Proposition \ref{yeti}. Scholze has shown in \cite[Cor.7.3]{scholze} that $$S(U^p, \OO)_{\mm}:=\varprojlim_n S(U^p, \OO/\varpi^n)_{\mm}, \quad \widehat{H}^1(U^p, \OO)_{\mm}:=\varprojlim_n \widehat{H}^1(U^p, \OO/\varpi^n)_{\mm}$$ have an action of the same Hecke algebra $\TT(U^p)_{\mm}$. 
Moreover, there is a surjective homomorphism of local rings $R_{\rhobar}\twoheadrightarrow \TT(U^p)_{\mm}$, where $R_{\rhobar}$ is the universal deformation ring of $\rhobar$.

\begin{thm}\label{thm_C} Let $x\in \mSpec \TT(U^p)_{\mm}[1/p]$ be such that the restriction of the corresponding 
Galois representation $\rho_x: G_{\QQ,S}\rightarrow \GL_2(\kappa(x))$ to $G_{\Qp}$ is irreducible. 
Then $(\widehat{H}^1(U^{p}, \OO)_\mm\otimes L)[\mm_x]$  is non-zero if and only if 
$(S(U^p, \OO)_{\mm}\otimes_\OO L )[\mm_x]$ is non-zero.

In this case, there is an isomorphism of admissible unitary $\kappa(x)$-Banach space representations
of $G_{\Qp}\times D_{p}^{\times}$: 
\begin{equation}\label{pups}
(\widehat{H}^1(U^{p}, \OO)_\mm\otimes_{\OO} L)[\mm_x]\cong \Shat^1(\Pi)^{\oplus n},
\end{equation}
where $\Pi$ is the absolutely irreducible 
$\kappa(x)$-Banach space representation corresponding to $\rho_x|_{G_{\Qp}}$ via the $p$-adic local Langlands correspondence for $\GL_2(\Qp)$.
In particular, the $\delta$-dimension of $(\widehat{H}^1(U^{p}, \OO)_\mm\otimes_\OO L)[\mm_x]$ is $1$.
\end{thm}

One should think of the first part of the theorem as a global Jacquet--Langlands correspondence between the $p$-adic automorphic forms on $D_0^{\times}$ and $D^{\times}$. In particular, since 
$D_0$ is split at $p$, one may always find a classical automorphic form, which is principal series at $p$, such that $(S(U^{p}, \OO)_\mm\otimes_{\OO} L)[\mm_x]$ is non-zero, where $\mm_x$ is the corresponding maximal ideal. The theorem tells us 
that there is a $p$-adic automorphic form on $D^{\times}$ corresponding to it. Such a form cannot be classical, since the classical Jacquet--Langlands correspondence cannot cope with principal series. 

It was pointed out to us by Sean Howe that in his 2017 University of Chicago PhD thesis, \cite{sean_howe}, he proves an analogue of the first part of the Theorem above in the setting when $D_0$ is a quaternion algebra over $\Q$ ramified at $p$ and $\infty$ and 
$D=\GL_2$, for the maximal ideals corresponding to overconvergent Galois representations. He conjectures  that the analogue of \eqref{pups} holds in his setting.

\begin{thm}\label{thm_D} Assume the setup of Theorem \ref{thm_C}. Let $D_{p}^{\times, 1}$ be the subgroup of $D_p^{\times}$ of elements with reduced norm equal to $1$. 
Let $\Shat^1(\Pi)^{\onealg}$ be the subspace of locally algebraic vectors for the action of $D_{p}^{\times, 1}$ on $\Shat^1(\Pi)$. Then $\Shat^1(\Pi)^{\onealg}$ is a finite dimensional $L$-vector space. If it is non-zero then $\rho_x|_{G_{\Qp}}$ is a twist of a potentially semistable representation, which does 
not become crystalline after restriction to the Galois group of any abelian extension of $\Qp$.
The quotient $\Shat^1(\Pi)/\Shat^1(\Pi)^{\onealg}$ contains an irreducible closed subrepresentation of $\delta$-dimension $1$.
\end{thm}
 
 To the best of our knowledge the existence of irreducible admissible unitary
Banach space representations of $D_p^\times$ of dimension $1$ has not been known before. 
The locally analytic representations of $D_p^\times$ constructed  by Kisin--Strauch in \cite{kisin_strauch} should have the dimension equal to the dimension of the $D_p^{\times}$-orbit 
in $\PP^1(\QQ_{p^2})$, which is equal to $\dim D_p^{\times} - \dim \QQ_{p^2}^\times= 4-2=2$. 

 If $F=\QQ$, as we assume in this introduction, then, after twisting, Theorem \ref{thm_D}  follows readily from  
 Theorem \ref{thm_C} by using Emerton's results on locally algebraic vectors in completed 
 cohomology, \cite{interpolate}. However, when $F$ is a totally real field, then one has 
 to deal with $p$-adic automorphic representations, which have locally algebraic vectors at all places above $p$, except for one, where they have locally algebraic vectors after a twist by a character, which is not 
 locally algebraic.

  Let us point out that because $D_{p}^\times$ has an open normal pro-$p$ subgroup every smooth irreducible representation of $D_{p}^{\times}$ in characteristic $p$ is finite dimensional 
  as a vector space. Moreover, if we fix a central character there are only finitely many isomorphism 
  classes. It follows from Theorem \ref{thm_B} that  at least one 
  of $H^1_{\et}(\PP^1_{\Cp}, \FF_{\pi_1})$ and $H^1_{\et}(\PP^1_{\Cp}, \FF_{\pi_2})$ must be an infinite dimensional $k$-vector space and admissible as smooth representations of $D_{p}^{\times}$. This means that they 
 are built together out of finite dimensional pieces in some non-semisimple way. This non-semisimplicity makes the study of the mod-$p$ Jacquet--Langlands correspondence complicated and passing to semi-simplification seems not to carry much information, since if the central character is fixed only finitely many isomorphism classes of irreducible subquotients can appear. 
 
 If we work in characteristic zero, then this problem need not appear. One may show that any irreducible unitary Banach space representation of $D_p^\times$ on a finite dimensional vector
 space is of the form $\Sym^b L^2 \otimes \det^a \otimes \tau\otimes \eta\circ \Nrd$, where
 $\tau$ is a smooth irreducible representation and $\eta:\Qp^{\times}\rightarrow L^\times$ is a 
 unitary character. Such representations appear in the classical Jacquet--Langlands correspondence (up to a twist) and we speculate that if the Galois representation corresponding to $\Pi$ does not become potentially semi-stable after twisting by a character then such representations 
should not appear as subquotients of $\Shat^1(\Pi)$.  It seems very reasonable to us in view of 
Corollary \ref{Cor_A} and Theorem \ref{thm_D}  to expect that $\Shat^1(\Pi)$ is of finite length 
as a Banach space representation of $D_p^{\times}$. This raises a natural question, whether
one can construct the irreducible representations in Theorem \ref{thm_D} directly, and prove local--global compatibility for them.  We hope to pursue these questions in future work.

 \subsection{Patching} We have decided not to use patching in this paper, since this would add another level of technicalities. However, let us indicate, which parts of the paper can be improved upon if one chooses
 to use patching. 
 
 In the miracle flatness theorem of Gee--Newton, see Proposition \ref{use_GN} for the form we use it, one needs the commutative ring to be regular, so we cannot apply it to 
 $\TT(U^p)_{\mm}$ and $S(U^p, \OO)_\mm$ directly, however if one patches under favourable assumptions one may arrange that the patched ring $R_\infty$ is regular, and then show that the patched module $M_\infty$ is flat over $R_\infty$ by the same argument. This would imply that $S(U^p, \OO)_\mm$
 is flat over $\TT(U^p)_{\mm}$. This is how Gee--Newton prove their big $R$ equals to big $\TT$ theorem. As a consequence one would know that the multiplicity $n$ in \eqref{pups} is independent of $x$, and one would not need to assume that $\rho_x|_{G_{\Qp}}$ is irreducible. 
 Instead we use Cohen's  structure theorem for complete local rings to prove flatness of $S(U^p, \OO)_\mm$  over some (random) formally smooth subring of $\TT(U^p)_{\mm}$
 of the same dimension. As a consequence we are  forced to assume that $\rho_x|_{G_{\Qp}}$ is irreducible, since we cannot exclude that at points corresponding to reducible 
 Galois representations, the fibre contains only one of the two irreducible Banach space that should appear there. Since we can not exclude that one of $H^1_{\et}(\PP^1_{\Cp}, \FF_{\pi_1})$ or $H^1_{\et}(\PP^1_{\Cp}, \FF_{\pi_2})$ vanishes, this causes trouble. However, since we don't need the precise knowledge of $(S(U^p, \OO)_\mm \otimes_\OO L)[\mm_x]$, but only that all irreducible subquotients of $\Pi$ appear there,  maybe arguments of Breuil--Emerton \cite{be} might be used here to get the same result without patching. 
 
 A further improvement could be made in Theorem \ref{thm_D}. Right now we use in an essential way that 
 $\Pi$ corresponds to the restriction to $G_{\Qp}$ of a global Galois representation. One could spread this result out by first showing that the patched module $M_\infty$ is projective as a $\GL_2(\Qp)$-representation 
 in a suitable category, as we do in \cite{6auth2}, so that if $r:G_{\Qp}\rightarrow \GL_2(L)$ is an irreducible representation satisfying $\rbarss=\chi_1\oplus \chi_2$ then the corresponding 
 $\GL_2(\Qp)$-representation $\Pi$ can be obtained by specializing $M_{\infty}$ at some closed point $y$ of $R_{\infty}[1/p]$. If $\sigma$ is a finite dimensional $D_p^\times$-invariant subspace
 of $\Shat^1(\Pi)$ then it is of the form 
 $\Sym^b L^2 \otimes \det^a \otimes \tau\otimes \eta\circ \Nrd$.
 Let $\sigma^\circ$ be an $D_p^\times$-invariant $\OO$-lattice in $\sigma$. Then $M_{\infty}'(\sigma^\circ):=\Hom_{D_p^{\times}}( \Shat^1(M_\infty), (\sigma^0)^d)^d$ is a finitely generated $R_\infty$-module.
 Arguing as in \cite[Lem.\,4.18]{6auth2} and using Theorem \ref{thm_D} one sees that the Galois representations corresponding to the closed points in $\mSpec R_{\infty}[1/p]$ in the support $M_{\infty}'(\sigma^{\circ})$ have the same $p$-adic Hodge theoretic properties, so for example if $\eta$ is trivial, then 
 all of them are potentially semistable, have the same Hodge--Tate weights, and the same inertial type, moreover the restriction of the associated Weil--Deligne representation to the Weil group of $\Qp$ is of "discrete series type". As a part of patching we obtain an ideal $\mathfrak a_\infty$ of $R_{\infty}$-generated by an $M_{\infty}$-regular sequence $y_1, \ldots, y_h$ such that $M_\infty/ \mathfrak a_{\infty} M_{\infty}\cong S(U^p, \OO)_\mm$. Then $\Shat^1(M_{\infty})/ \mathfrak a_{\infty}\Shat^1(M_{\infty}) \cong \widehat{H}^1(U^p, \OO)_{\mm}$ and $\varpi, y_1, \ldots, y_h$ 
 is a system of parameters for $M_\infty'(\sigma^\circ)$. Thus $M_\infty'(\sigma^\circ)/\mathfrak a_{\infty} M_\infty'(\sigma^\circ) [1/p]$ is non-zero and we may convert this into a representation 
 theoretic information using \cite[Prop.\,2.22]{duke}, to deduce that there is some $x\in \mSpec \TT(U^p)[1/p]$, such that $\Hom_{D_p^\times}(\sigma, (\widehat{H}^1(U^p, \OO)_\mm\otimes_\OO L)[\mm_x])\neq 0$. Then as in the proof of Theorem \ref{main_2} the $p$-adic Hodge theoretic data of  $\rho_x|_{G_{\Qp}}$  (and hence of $r$) will determine $\sigma$. This allows us to conclude
 that if $\Hom_{D_p^\times}(\sigma, \Shat^1(\Pi))$ is non-zero then $\Shat^1(\Pi)^{\onealg}$ is isomorphic to 
 a finite direct sum of copies of $\sigma$ and thus is a finite dimensional $L$-vector space. Arguing as in the proof of Theorem \ref{main_2} we conclude that $\Shat^1(\Pi)/\Shat^1(\Pi)^{\onealg}$ 
is non-zero and any irreducible Banach space subrepresentation has $\delta$-dimension equal to $1$.  The details of this argument will appear in a subsequent  paper.
 \subsection{Outline}
 We end the introduction with a brief outline how the paper is structured.  Section \ref{local} contains 
 the local part of the paper. We first review the results of Ludwig and Scholze, we then recall some
 of results about the representation theory of $\GL_2(\Qp)$ proved in \cite{image}. In subsection 
 \ref{dualcat} in order to make things transparent for the reader we dualize everything at least twice.
 In section \ref{fandf} we recall the notion of dimension for finitely generated modules over Auslander
 regular rings and  the results  of Gee--Newton on miracle flatness in  non-commutative setting.
 In section \ref{lg} we prove the local-global compatibility results. The main objective of this section 
 is to be able to apply the results of \cite{image} to the completed cohomology. These results were known to the authors of \cite{6auth2} at the time of writing that paper. However, the results proved in section \ref{lg} form a technical backbone of the paper, so it seems like a good idea to write down the details. The main ingredients are the theory of capture \cite[\S 2.4]{CDP},
\cite[\S 2.1]{blocksp2}, which is based on the ideas of Emerton in \cite{emerton_lg}, the results of Berger--Breuil \cite{bb} on the universal unitary completions of locally algebraic principal series representations and the results of Colmez in \cite{colmez} on the compatibility of the $p$-adic and classical local Langlands correspondence. In section \ref{section_main} we prove the theorems stated in the introduction. 

\subsection{Acknowledgements} I thank Judith Ludwig for giving a nice talk at the Oberseminar in Essen and for subsequent discussions. I thank Peter Scholze for his comments during and after my talk in Bonn in October 2017. I have first presented the results in the Morning Side Center for Mathematics 
in Beijing in March 2017, and finished writing up the paper there during my visit in March 2018. I thank Yongquan Hu for the invitation and stimulating discussions, I thank the Morning Side Center for Mathematics for providing excellent working conditions. The comments of Yiwen Ding made during my talk at Peking University helped to improve the exposition of the paper. I thank Jan Kohlhaase for sharing his insights into dimension theory of Auslander regular rings with me and Shu Sasaki for discussions about Shimura curves. I thank  Konstantin Ardakov, Matthew Emerton and Toby Gee  for helpful correspondence. I thank Florian Herzig, Judith Ludwig and James Newton for their comments on an earlier draft. The author is partially supported by the DFG, SFB/TR45.

\section{Notation}\label{notation} Our conventions on local class field theory, classical local Langlands correspondence and $p$-adic Hodge theory agree with those of \cite{6auth2}. In particular, 
uniformisers correspond to geometric Frobenii and the $p$-adic cyclotomic character $\varepsilon: G_{\Qp}\rightarrow \Zp^\times$ has Hodge--Tate weight $-1$. We will denote its 
reduction modulo $p$ by $\omega$. We will denote by $\chi_{\cyc}: G_{\QQ}\rightarrow \Zp^\times$  the global $p$-adic cyclotomic character. 

If $G$ is a $p$-adic analytic group then we employ the notation scheme introduced in \cite{ord1} for 
the categories of its representations. In particular, $\Mod^{\sm}_G(\OO)$ denotes the category of smooth representations of $G$ on $\OO$-torsion modules, $\Mod^{\adm}_G(\OO)$ denotes the full subcategory of admissible representations and $\Mod^{\ladm}_G(\OO)$ denotes the full 
subcategory of locally admissible representations of $G$.

 If $\zeta: Z(G)\rightarrow \OO^{\times}$ is a continuous character of the centre of $G$ then we add a subscript $\zeta$ to indicate that we consider only those representations on 
which $Z(G)$ acts by the central character $\zeta$. For example, $\Mod^{\ladm}_{G, \zeta}(\OO)$ is the full subcategory of $\Mod^{\sm}_G(\OO)$ consisting of locally admissible 
representations on which $Z(G)$ acts by the character $\zeta$. 

The functor $M\mapsto M^{\vee}:=\Hom^{\cont}_{\OO}(M, L/\OO)$ induces an anti-equivalence
of categories between the category of discrete $\OO$-modules and the category of compact 
$\OO$-modules. We will refer to this duality as Pontryagin duality and to $M^{\vee}$ as the Pontryagin dual of $M$. 

Pontryagin duality induces an anti-equivalence of categories between $\Mod^{\sm}_G(\OO)$ and a category, which Emerton calls profinite augmented $G$-representations over $\OO$, 
see \cite[Def.\,2.1.6]{ord1}, and we will denote by $\Mod^{\pro}_G(\OO)$. We will denote the 
full subcategory of $\Mod^\pro_G(\OO)$ anti-equivalent to $\Mod^{\sm}_{G, \zeta}(\OO)$ by Pontryagin duality by 
$\Mod^{\pro}_{G, \zeta}(\OO)$. In particular, if $G$ is compact then $\Mod^{\pro}_{G, \zeta}(\OO)$ is the category of compact $\OO[[G]]$-mo\-du\-les, on which the centre acts by $\zeta^{-1}$, 
where $\OO[[G]]$ is the completed group algebra of $G$. 

We would like to review Schikhof duality, which we learned from \cite{iw}. Let $\Pi$ be an $L$-Banach space and let $\Theta$ be an open bounded lattice in $\Pi$.  Its Schikhof dual $\Theta^d$ is defined as $\Hom_{\OO}^{\cont}(\Theta, \OO)$ equipped with the weak topology (or the topology of point-wise convergence). It is $\OO$-torsion free and 
compact, see the proof of \cite[Thm.\,1.2]{iw}. We also let  $\Pi^d:=\Hom_L^{\cont}(\Pi, L)$  equipped with the weak topology. It follows from \cite[Prop.\,3.1]{nfa} that $\Pi^d \cong \Theta^d\otimes_{\OO} L$. 

 Conversely, if $M$  is a linearly compact, torsion free $\OO$-module then 
 $\Hom^{\cont}_{\OO}(M, L)$, equipped with a supremum norm 
is an $L$-Banach space with the unit ball $M^d:= \Hom^{\cont}_{\OO}(M, \OO)$. 
Note that $M^d$ is complete for the $p$-adic topology. Since $M$ is $\OO$-torsion free, it is projective in the category 
of linearly compact $\OO$-modules, see \cite[Rem.\,1.1]{iw}. From this one obtains an isomorphism 
\begin{equation}\label{Mdpn}
M^d/\varpi^n \cong (M/\varpi^n)^{\vee},
\end{equation} 
see the proof of   \cite[Lem.\,5.4]{comp}. We thus obtain a homeomorphism 
\begin{equation}\label{Md}
M^d \cong \varprojlim_n M^d/\varpi^n \cong \varprojlim_n (M/\varpi^n)^{\vee}.
\end{equation}
It follows from \cite[Rem.\,10.2]{nfa} applied to $\Pi$ and the gauge of $\Theta$ that $\Theta/\varpi^n$ 
is a free $\OO/\varpi^n$-module with basis indexed by a choice of $k$-vector space basis of $\Theta/\varpi$. 
Thus $\Hom_{\OO}(\Theta/\varpi^n, \OO/\varpi^n)\rightarrow \Hom_{\OO}(\Theta/\varpi^n, \OO/\varpi^m)$ is surjective 
for $n\ge m$. By passing to the limit we deduce that $\Hom_{\OO}^{\cont}(\Theta, \OO)\rightarrow \Hom_{\OO}^{\cont}(\Theta, 
\OO/\varpi^n)$ is surjective and induces a homeomorphism 
\begin{equation}\label{Thetadpn}
\Theta^d/\varpi^n \cong (\Theta/\varpi^n)^{\vee}.
\end{equation}
 It follows from \eqref{Md} applied to $M=\Theta^d$ and \eqref{Thetadpn} that
\begin{equation}\label{Thetad}
(\Theta^d)^d\cong \varprojlim_n (\Theta^d/\varpi^n)^{\vee}\cong\varprojlim_n ((\Theta/\varpi^n)^{\vee})^{\vee}\cong 
\varprojlim_n \Theta/\varpi^n \cong \Theta.
\end{equation}
If $\Theta=M^d$ then \eqref{Thetadpn} and \eqref{Mdpn} give
\begin{equation} 
\Theta^d/\varpi^n\cong (M^d/\varpi^n)^{\vee}\cong ((M/\varpi^n)^{\vee})^{\vee}\cong M/\varpi^n.
\end{equation}
Since $\Theta^d$ and $M$  are compact $\OO$-modules we obtain a homeomorphism
\begin{equation}
(M^d)^d \cong \varprojlim_n \Theta^d/\varpi^n \cong \varprojlim_n M/\varpi^n \cong M.
\end{equation}

We are most interested in the situation, when $\Pi$ is a unitary $L$-Banach space representation of 
a $p$-adic analytic group $G$. In this case $\Theta^d$ is a topological $\OO[[K]]$-module, where $K$ 
is a compact open subgroup of $G$, see \cite[Sect.\,2]{iw}. We say that $\Pi$ is \textit{admissible}
if $\Theta^d$ is finitely generated as $\OO[[K]]$-module. The functor $\Pi \mapsto \Pi^d$ induces 
an anti-equivalence of categories between the category of admissible unitary $L$-Banach space representations of $K$ and
the category of finitely generated $\OO[[K]]\otimes_{\OO} L$-modules, \cite[Thm.\,3.5]{iw}.

Colmez in \cite[\S IV]{colmez} has defined an exact covariant  functor $\VV$ from the category of 
finite length smooth representations of $G:=\GL_2(\Qp)$ with a central character
on $\OO$-torsion modules to the category of smooth finite length representations of $G_{\Qp}$ on $\OO$-torsion modules. We modify this functor as follows. Let 
$\dualcat_{\zeta}(\OO)$ be the full subcategory of $\Mod^{\pro}_G(\OO)$ anti-equivalent to 
$\Mod^{\ladm}_{G, \zeta}(\OO)$ by the Pontryagin duality. If $M$ is in $\dualcat_{\zeta}(\OO)$
then we may write $M=\varprojlim_n M_n$, where the projective limit is taken over all quotients of 
finite length. We then define $\cV(M):= \varprojlim_n \VV(M_n^{\vee})^{\vee} (\zeta)$, where 
we view $\zeta$ as a character of $G_{\Qp}$ via local class field theory. This normalisation
differs from \cite{image} by a twist of cyclotomic character and coincides with the normalisation 
of \cite{6auth2}. In particular, if $\pi$ is a principal series representation $\Ind_B^G \omega \chi_1\otimes \chi_2$ then $\cV(\pi^{\vee})=\chi_1$ viewed as a character of $G_{\Qp}$ via the local 
class field theory. 

Let $\Pi$ be an admissible unitary $L$-Banach space representation of $G$ with central character 
$\zeta$ and let  $\Theta$ be an open bounded $G$-invariant lattice  in $\Pi$.  It is easy to see that $\Theta^d$ is an object of $\dualcat_\zeta(\OO)$, see 
\cite[Lem.\,4.11]{image}. Thus we may apply the functor $\cV$ to $\Theta^d$ to obtain a continuous 
$G_{\Qp}$-representation on a compact $\OO$-module. We define 
$\cV(\Pi):=\cV(\Theta^d)\otimes_{\OO} L$. The definition of $\cV(\Pi)$ does not depend on the choice of $\Theta$, since any two are commensurable. The functor $\Pi \mapsto \cV(\Pi)$ is contravariant. If $\Pi$ is absolutely irreducible and  occurs as a subquotient of 
a unitary parabolic induction of a unitary character, then we say that $\Pi$ is \textit{ordinary}. 
Otherwise, we say that $\Pi$ is \textit{non-ordinary}. In this case it is shown in \cite{image}, \cite{CDP} that $\cV(\Pi)$ is a $2$-dimensional representation of $G_{\Qp}$ and (taking into account our normalisations) $\det \cV(\Pi)= \zeta \varepsilon^{-1}$. A deep theorem of Colmez proved in \cite{colmez} relates the existence of locally algebraic vectors in $\Pi$ to the property of $\cV(\Pi)$
being potentially semi-stable with distinct Hodge--Tate weights. With our conventions Colmez's result says that $\Hom_U( \det^a \otimes \Sym^b L^2, \Pi)\neq 0$ for some open subgroup $U$ of $\GL_2(\Zp)$
if and only if $\cV(\Pi)$ is potentially semi-stable with Hodge--Tate weights  $(1-a, -a-b)$. 

If $A$ is a commutative ring then we denote by $\mSpec A$ the set of its maximal ideals. If
$x\in \mSpec A$ then $\kappa(x)$ will denote its residue field. We will typically be considering $\mSpec A$ when 
$A=R[1/p]$, where $R$ is a complete local noetherian $\OO$-algebra with residue field $k$. In this 
case $\kappa(x)$ is a finite extension of $L$.

\section{Local part}\label{local}
 Let $F$ be a finite extension of $\Qp$. We fix an algebraic closure $\Fbar$ of $F$ and let $G_F=\Gal(\Fbar/ F)$. Let $\Cp$ be the completion of $\Fbar$ and 
let $\Fbreve$ be the completion of the maximal unramified extension of $F$ in $\Fbar$. Let $G=\GL_n(F)$ and let $D/F$ be a central division algebra over $F$ with invariant $1/n$. 

\subsection{Results of Ludwig and Scholze}\label{ls}
We continue to denote by $L$ a (sufficiently large) finite extension of $F$ with the ring of integers $\OO$, uniformiser $\varpi$ and residue field $k$. To a smooth representation $\pi$ of $G$ on an 
$\OO$-torsion  module, Scholze associates a Weil-equivariant sheaf $\mathcal F_{\pi}$ on the \'etale site of the adic space $\PP^{n-1}_{\Fbreve}$, see \cite[Prop. 3.1]{scholze}. If $\pi$ is admissible 
then he shows that for any $i\ge 0$ the \'etale cohomology groups $H^i_\et(\PP^{n-1}_{\Cp}, \FF_{\pi})$ carry a continuous $D^{\times}\times G_F$-action, which make them into 
 smooth admissible representations of $D^{\times}$.  Moreover, they vanish for $i> 2(n-1)$, see \cite[Thm. 3.2, 4.4]{scholze}. His construction is expected to realize both the $p$-adic Jacquet-Langlands and the $p$-adic local Langlands correspondences. The trouble is that these groups seem to be impossible to calculate in most cases. It is known that the natural map 
 \begin{equation}\label{H0} 
 H^0_\et(\PP^{n-1}_{\Cp}, \FF_{\pi^{\SL_n(F)}})\hookrightarrow H^0_\et(\PP^{n-1}_{\Cp}, \FF_{\pi})
 \end{equation}
 is an isomorphism, see \cite[Prop.\,4.7]{scholze}. In particular,  if $\pi^{\SL_n(F)}=0$ then $H^0$ vanishes. 

We want to extend the results of Scholze to the category of locally admissible representations. Recall  that a smooth representation of $G$ or $D^{\times}$ on an $\OO$-torsion module is \textit{locally admissible} if it is equal to the union of its admissible subrepresentations. In this
case we may write  $\pi=\varinjlim \pi'$, where the limit is taken over all admissible subrepresentations of $\pi$. The proof of  \cite[Prop.3.1]{scholze} shows that the natural map 
$$ \varinjlim \FF_{\pi'} \rightarrow \FF_{\pi},$$
is an isomorphism, since it induces an isomorphism on stalks at geometric points. In our setting cohomology commutes with direct limits. For all $i\ge 0$ we have isomorphisms
\begin{equation}\label{limits}
\begin{split}
H^i_\et(\PP^{n-1}_{\Cp}, \FF_{\pi})&\cong \varinjlim_K  H^i( (\PP^{n-1}_{\Cp}/K)_{\et}, \FF_{\pi})\\& \cong \varinjlim_K \varinjlim_{\pi'} H^i( (\PP^{n-1}_{\Cp}/K)_{\et}, \FF_{\pi'})\\& \cong \varinjlim_{\pi'} H^i_\et(\PP^{n-1}_{\Cp}, \FF_{\pi'}),
\end{split}
\end{equation}
where the first isomorphism follows from \cite[Prop.2.8]{scholze} and the limit is taken over all compact open subgroups of $D^{\times}$, the second isomorphism follows from the fact that the site $(\PP^{n-1}_{\Cp}/K)_{\et}$ is coherent, see \cite[Lem.\,2.7(v)]{scholze}, so cohomology commutes with filtered direct limits \cite[Exp.VI, Cor.5.2]{sga4}, the last isomorphism is \cite[Prop.2.8]{scholze} again. Since quotients of admissible representations of $p$-adic analytic groups are admissible, from  \eqref{limits} and Scholze's results 
we deduce that $H^i_\et(\PP^{n-1}_{\Cp}, \FF_{\pi})$ is a locally admissible representation of $D^{\times}$, which vanishes for $i> 2(n-1)$. 

If $n=2$ then the cohomology vanishes for $i>2$. If we additionally assume that $F=\Qp$ and $\pi$ is a principal series representation Ludwig has shown in \cite[Thm.4.6]{ludwig} that $H^2_\et(\PP^{1}_{\Cp}, \FF_{\pi})$ vanishes. Thus we get the following:

\begin{cor}\label{cor_ludwig} If $\pi$ is a locally admissible representation of $\GL_2(\Qp)$, such that all its irreducible subquotients are principal series, then 
$$H^i_{\et}(\PP^{1}_{\Cp}, \FF_{\pi})=0, \quad \forall i\neq 1.$$
\end{cor}
\subsection{Representation theory of $\GL_2(\Qp)$}
From now on we assume that $n=2$ and $F=\Qp$ so that $G=\GL_2(\Qp)$. Let us recall some representation theory of $G$. The category $\Mod^\ladm_G(\OO)$ of smooth locally admissible $G$-representations on $\OO$-torsion modules decomposes into 
a direct sum of indecomposable subcategories called blocks. Blocks containing an absolutely irreducible $k$-representation of $G$ correspond to semi-simple representations $\rbarss: G_{\Qp}\rightarrow \GL_2(k)$, such that  $\rbarss$ is either absolutely irreducible or a direct sum of $1$-dimensional representations. 

Let us describe these blocks explicitly. We denote by $[\pi]$ the isomorphism class of a representation 
$\pi$ of $G$. If $\rbarss$ is absolutely irreducible then let $\BB_{\rbarss}=\{[\pi]\}$, where $\pi$ is the supersingular representation of $G$ corresponding to $\rbarss$ under the semi-simple mod $p$ local Langlands correspondence, \cite{breuil1}. If $\rbarss= \chi_1\oplus\chi_2$ then we consider $\chi_1$ and $\chi_2$ as characters of $\Qp^{\times}$ via the local class field theory, and let 
\begin{equation}\label{def_pi}
\pi_1:=\Ind_B^G \chi_1\omega\otimes \chi_2, \quad \pi_2:=\Ind_B^G \chi_2\omega \otimes \chi_1,
\end{equation}
where we let $\varepsilon: \Qp^{\times}\rightarrow L^\times$ be the character, which 
maps $x$ to $x|x|$, and let $\omega: \Qp^{\times}\rightarrow k^{\times}$ be its reduction modulo $\varpi$. 
We then let $\BB_{\rbarss}$ be the set of isomorphism classes of irreducible subquotients of $\pi_1$ and $\pi_2$. It can have from one up to four elements depending on $\chi_1 \chi_2^{-1}$. 

Let $\Mod^{\ladm}_G(\OO)_{\rbarss}$ be the full subcategory of the category of locally admissible representations $\Mod^{\ladm}_G(\OO)$ of $G$, 
such that $\pi$ is in $\Mod^{\ladm}_G(\OO)_{\rbarss}$ if and only if the isomorphism classes of all the irreducible subquotients of $\pi$ lie in $\BB_{\rbarss}$. It follows from the $\Ext^1$-calculations in \cite[Cor.\,1.2]{durham} and 
\cite[Prop.\,5.34]{image} that $\Mod^{\ladm}_G(\OO)_{\rbarss}$ is a direct summand of the category $\Mod^{\ladm}_G(\OO)$, in the sense that every $\pi\in \Mod^{\ladm}_G(\OO)$ can be written uniquely as 
\begin{equation}\label{decompose_can}
 \pi=\pi_{\rbarss} \oplus \pi^{\rbarss},
 \end{equation}
where $\pi_{\rbarss}$ is the maximal $G$-invariant subspace of $\pi$, such that the isomorphism classes of all 
its irreducible subquotients lie in $\BB_{\rbarss}$, and $\pi^{\rbarss}$ is the maximal $G$-invariant subspace of 
$\pi$ such that  none of its irreducible subquotients lie in $\BB_{\rbarss}$.

In this paper we are especially interested in the case, when $\rbarss=\chi_1\oplus \chi_2$ and 
$\chi_1\chi_2^{-1}\neq \omega^{\pm 1}$. In this case both representations $\pi_1$ and $\pi_2$ in \eqref{def_pi} are irreducible principal series representations and  so every $\pi\in \Mod^\ladm_G(\OO)_{\rbarss}$ satisfies the hypothesis of Corollary \ref{cor_ludwig}. Let $\Mod^\ladm_{G_{\Qp} \times D^\times}(\OO)$ be the category of locally admissible representations 
of $D^\times$ on $\OO$-torsion modules with a continuous commuting $G_{\Qp}$-action.
We immediately deduce 

\begin{cor}\label{ludwig_cor} If $\rbarss=\chi_1\oplus \chi_2$ and 
$\chi_1\chi_2^{-1}\neq \omega^{\pm 1}$ then the functor 
$$\Sc^1: \Mod^\ladm_G(\OO)_{\rbarss} \longrightarrow \Mod^\ladm_{G_{\Qp} \times D^\times}(\OO), \quad
 \pi\mapsto H^1_{\et}(\PP^{1}_{\Cp}, \FF_{\pi}).$$
is exact and covariant.
\end{cor}

\subsection{Dual category and Banach space representations}\label{dualcat}
Let $\dualcat_G(\OO)$ be the category anti-equivalent to $\Mod^{\ladm}_G(\OO)$
via Pontryagin duality and let $\dualcat_{G_{\Qp} \times D^\times}(\OO)$ be the category anti-equivalent to $\Mod^\ladm_{G_{\Qp} \times D^\times}(\OO)$ via the Pontryagin duality. We define a covariant homological $\delta$-functor $\{ \Shat^i\}_{i\ge 0}$ by 
$$\Shat^i: \dualcat_G(\OO)\rightarrow \dualcat_{G_{\Qp} \times D^\times}(\OO), \quad
 M\mapsto H^i_{\et}(\PP^{1}_{\Cp}, \FF_{M^{\vee}})^{\vee}$$
where $M^\vee=\Hom^{\cont}_{\OO}(M, L/\OO)$ denotes the Pontryagin dual of $M$. Note that 
$(M^{\vee})^{\vee}\cong M$.  We introduce these dual categories, because it is much more convenient to work with compact torsion-free $\OO$-modules than with discrete divisible $\OO$-modules.

We denote the category of unitary admissible $L$-Banach space representations 
of $G$ by $\Ban^{\adm}_G(L)$. If $\Pi$ is in $\Ban^{\adm}_G(L)$ and $\Theta$ is an open bounded $G$-invariant lattice in $\Pi$ then it follows from \cite[Lemma 4.4]{image}  that the Schikhof dual 
$$\Theta^d:=\Hom^{\cont}_{\OO}(\Theta, \OO)$$
equipped with the weak topology is an object of $\dualcat_G(\OO)$. We refer the reader to 
Section \ref{notation} for properties of the Schikhof dual. We thus may apply the functors $\Shat^i$ to it to obtain a compact $\OO$-module $\Shat^i(\Theta^d)$ with a continuous  $G_{\Qp} \times D^\times$-action. We denote the Schikhof dual of this module by $\Shat^i(\Theta^d)^d$ and equip it with the $p$-adic topology. 

\begin{lem}\label{trouble} There is an exact sequence of topological $\OO$-modules
$$0\rightarrow \varprojlim_m (H^{i-1}_\et(\PP^{1}_{\Cp}, \FF_{\Theta\otimes_\OO L/\OO})/\varpi^m)\rightarrow \varprojlim_m H^i_{\et}(\PP^{1}_{\Cp}, \FF_{\Theta/\varpi^m})\rightarrow \Shat^i(\Theta^d)^d\rightarrow 0,$$
which identifies $\Shat^i(\Theta^d)^d$ with the maximal Hausdorff $\OO$-torsion free quotient of 
$\varprojlim_m H^i_{\et}(\PP^{1}_{\Cp}, \FF_{\Theta/\varpi^m})$.
\end{lem}
\begin{proof}  Since 
\begin{equation}\label{modpin}
\Theta^d/\varpi^n \Theta^d\cong (\Theta/\varpi^n\Theta)^{\vee}, \quad \forall n\ge 1,
\end{equation}
 we have natural isomorphism $(\Theta^d)^\vee\cong \Theta\otimes_\OO L/\OO$ and thus
 \begin{equation}\label{covid-19}
  \Shat^i(\Theta^d)\cong H^i_\et(\PP^{1}_{\Cp}, \FF_{\Theta\otimes_\OO L/\OO})^{\vee}\cong (\varinjlim_n 
 H^i_\et(\PP^{1}_{\Cp}, \FF_{\Theta/\varpi^n}))^{\vee}.
 \end{equation}
Hence
\begin{multline*}
 \Shat^i(\Theta^d)^d\cong \varprojlim_m \Hom^\cont_{\OO}(\Shat^i(\Theta^d), \OO/\varpi^m)\cong 
\varprojlim_m (\Shat^i(\Theta^d)/\varpi^m)^\vee\\
\cong \varprojlim_m ((H^i_\et(\PP^{1}_{\Cp}, \FF_{\Theta\otimes_\OO L/\OO})[\varpi^m])^{\vee})^{\vee})\cong \varprojlim_m (H^i_\et(\PP^{1}_{\Cp}, \FF_{\Theta\otimes_\OO L/\OO})[\varpi^m]),
\end{multline*}
where the transition maps in the projective system are induced by multiplication by $\varpi$. 
The short exact sequence $0\rightarrow \Theta/\varpi^m \rightarrow \Theta\otimes_\OO L/\OO
\overset{\varpi^m}{\rightarrow }
\Theta\otimes_\OO L/\OO \rightarrow 0$ yields an exact sequence: 
$$ 0\rightarrow H^{i-1}_\et(\PP^{1}_{\Cp}, \FF_{\Theta\otimes_\OO L/\OO})/\varpi^m\rightarrow H^i_{\et}(\PP^{1}_{\Cp}, \FF_{\Theta/\varpi^m}) \rightarrow 
H^{i}_\et(\PP^{1}_{\Cp}, \FF_{\Theta\otimes_\OO L/\OO})[\varpi^m]\rightarrow 0.$$
The exact sequence is obtained  by passing to the limit by noting that the system satisfies the 
Mittag-Leffler condition.

If we let $M:=H^{i-1}_\et(\PP^{1}_{\Cp}, \FF_{\Theta\otimes_\OO L/\OO})$, $\widehat{M}$ its $p$-adic completion and $M'$ the image of 
$M$ in $\widehat{M}$, then $M'$ is a dense $\OO$-torsion submodule of $\widehat{M}$. We equip $\widehat{M}/M'$ with the quotient topology. Then for all topological $\OO$-modules
 $N$ which are Hausdorff and $\OO$-torsion free we have
$$\Hom^{\cont}_{\OO}(\widehat{M}, N)\cong \Hom^{\cont}_{\OO}(\widehat{M}/M', N)=0.$$
Since $\Shat^i(\Theta^d)^d$ is Hausdorff and $\OO$-torsion free this implies the last assertion.
\end{proof}
We define 
$$ \Shat^i(\Pi):= \Shat^i(\Theta^d)^d\otimes_{\OO} L.$$
The definition does not depend on the choice of $\Theta$, since any two are commensurable. 
To motivate this definition we  observe that \eqref{modpin} implies that 
 we have natural isomorphisms 
 $$((\Theta\otimes_\OO L/\OO)^{\vee})^d\cong \Theta, \quad ((\Theta\otimes_\OO L/\OO)^{\vee})^d\otimes_\OO L\cong \Pi.$$

\begin{lem}\label{adm_banach} If $\Pi$ is an admissible unitary $L$-Banach space representation of $G$ then $\Shat^i(\Pi)$ is an admissible unitary $L$-Banach space representation of $D^{\times}$ for all $i\ge 0$. 
\end{lem}
\begin{proof} By applying $\{ \Shat^i\}_{i\ge 0}$ to the exact sequence $0\rightarrow \Theta^d \overset{\varpi}{\rightarrow} \Theta^d \rightarrow \Theta^d/ \varpi \Theta^d\rightarrow 0$ we obtain a long exact sequence: 

\begin{multline}\label{long_exact}
0\rightarrow  \Shat^2(\Theta^d)\overset{\varpi}{\rightarrow} \Shat^2(\Theta^d)\rightarrow \Shat^2(\Theta^d/\varpi)\rightarrow \Shat^1(\Theta^d)\overset{\varpi}{\rightarrow} \Shat^1(\Theta^d)\\\rightarrow  \Shat^1(\Theta^d/\varpi)\rightarrow \Shat^0(\Theta^d)\overset{\varpi}{\rightarrow} \Shat^0(\Theta^d)\rightarrow 
 \Shat^0(\Theta^d/\varpi)\rightarrow 0.
 \end{multline}

The terms for $i\ge 3$ vanish, because of the results of Scholze explained in the previous section. 
Moreover, since $\Pi$ is an admissible Banach space representation we know that $\Theta/\varpi$ is an admissible smooth $G$-representation. Scholze's result implies that $H^i_{\et}( \PP^{1}_{\Cp}, \FF_{\Theta/\varpi})$ are admissible smooth $D^{\times}$-representations. Thus the Pontryagin dual is a finitely generated $\OO[[K]]$-module, where $K$ is any compact open subgroup of $D^{\times}$. Since 
$\OO[[K]]$ is noetherian we deduce from \eqref{long_exact} that $\Shat^i(\Theta^d)/ \varpi$ is a finitely generated $\OO[[K]]$-module and topological Nakayama's lemma for compact $\OO[[K]]$-modules implies that $\Shat^i(\Theta^d)$ is a finitely generated $\OO[[K]]$-module, see \cite[Lem.\,1.4, Cor.\,1.5]{brumer}. The assertion follows from
the theory of Schneider-Teitelbaum \cite[Thm.\,3.5]{iw}.
\end{proof}

Now let us note, that even if we could  show that $\Shat^1(\Theta^d/\varpi)$ is non-zero, we cannot 
rule out using \eqref{long_exact} alone that $\Shat^i(\Pi)=0$ for all $i\ge 0$. A priori it could happen that 
$\Shat^2(\Theta^d)$ and $\Shat^0(\Theta^d)$ are both zero and $\Shat^1(\Theta^d)$ is killed by some power of $\varpi$. This is where Ludwig's theorem enters.

Let $\rbarss$ be  as in the previous section. Let $\Ban^{\adm}_G(L)_{\rbarss}$ be the full subcategory of the category of $\Ban^{\adm}_G(L)$ such that $\Pi$ is in $\Ban^{\adm}_G(L)_{\rbarss}$ if and only if 
 the isomorphism classes of all the irreducible subquotients of $\Theta/\varpi$ lie in $\BB_{\rbarss}$. We note that this 
 last condition depends only on $\Pi$ and not on the choice of open bounded $G$-invariant lattice $\Theta$, see \cite[Lem.\,4.3]{image}.
  One can deduce from the result for $\Mod^{\ladm}_G(\OO)$  that $\Ban^{\adm}_G(L)_{\rbarss}$ is a direct summand of $\Ban^{\adm}_G(L)$, see \cite[Prop.\,5.36]{image}.
 
 \begin{prop}\label{enough} Assume that $\rbarss=\chi_1\oplus \chi_2$ with $\chi_1\chi_2^{-1}\neq \omega^{\pm 1}$
 and let $\pi_1$ and $\pi_2$ be the principal series representation defined in \eqref{def_pi}, so that
  $\BB_{\rbarss}$ consists of the isomorphism classes of $\pi_1$ and $\pi_2$. Then the following assertions are equivalent: 
  \begin{itemize}
  \item[(i)] both $H^1_{\et}(\PP^{1}_{\Cp}, \FF_{\pi_1})$ and $H^1_{\et}(\PP^{1}_{\Cp}, \FF_{\pi_2})$ vanish;
  \item[(ii)] $H^1_{\et}(\PP^1_{\Cp}, \FF_{\pi})=0$ for all $\pi$ in $\Mod^{\ladm}_G(\OO)_{\rbarss}$;
  \item[(iii)] $\Shat^1(\Pi)=0$ for all  $\Pi\in \Ban^{\adm}_{G}(L)_{\rbarss}$;
  \item[(iv)] if $\Pi\in \Ban^{\adm}_{G}(L)_{\rbarss}$ is  absolutely irreducible, non-ordinary then 
  $\Shat^1(\Pi)=0$. 
  \end{itemize}
  \end{prop}
  \begin{proof} Any locally admissible representation of $\GL_2(\Qp)$ is equal to the union of its subrepresentations of finite length, see \cite[Thm.\,2.3.8]{ord1}. 
  If (i) holds then (ii) follows from \eqref{limits} and Corollary \ref{ludwig_cor}. If (ii) holds then \eqref{covid-19} implies (iii), 
  which trivially implies (iv). For the proof of (iv) implies (i) recall that $\Pi$ is non-ordinary if and only if 
  it does not occur as a subquotient of a unitary parabolic induction of any unitary character of the maximal torus in $G$. If $\Theta$ is an open bounded $G$-invariant lattice in $\Pi$ then the semi-simplification of $\Theta/\varpi$ is isomorphic to $\pi_1\oplus \pi_2$ by \cite[Thm.\,11.1]{image}. Ludwig's theorem implies that $H^2(\PP^1_{\Cp}, \FF_{\Theta/\varpi})=0$ and the isomorphism \eqref{H0} implies that the same holds for $H^0$. Topological Nakayama's lemma together with \eqref{long_exact} implies that $\Shat^i(\Theta^d)=0$ for $i=0$ and $i=2$. We deduce that we have a short exact sequence:
 \begin{equation}\label{sese}
 0 \rightarrow \Shat^1(\Theta^d)\overset{\varpi}{\rightarrow} \Shat^1(\Theta^d)\rightarrow  
  H^1_{\et}(\PP^1_{\Cp}, \FF_{\Theta/\varpi})^{\vee}\rightarrow 0.
 \end{equation}
 Thus  $\Shat^1(\Theta^d)$ is $\OO$-torsion free. Since $\Shat^1(\Theta^d)$ is a compact $\OO$-module, this implies that $\Shat^1(\Theta^d)$ is isomorphic to a product of copies of $\OO$. In particular, 
 if $\Shat^1(\Theta^d)\neq 0$ then $\Shat^1(\Pi)\neq 0$. Thus (iv) implies that $\Shat^1(\Theta^d)=0$ and 
 \eqref{sese} implies that $H^1_{\et}(\PP^1_{\Cp}, \FF_{\Theta/\varpi})=0$. Since the semisimplification 
 of $\Theta/\varpi$ is isomorphic to $\pi_1\oplus \pi_2$, Corollary \ref{ludwig_cor} implies (i).
 \end{proof}
 
 \begin{remar} We will show later on that part (i) of the Proposition does not hold by showing that completed cohomology gives a counterexample to (ii). However, we 
 can not rule out that   one of the 
 groups can  vanish in (i) (unless of course $\pi_1\cong \pi_2$). Most likely both groups are non-zero
 since there is no natural way to distinguish one of the principal series in the block.
 \end{remar}
 
 Let us record a further consequence of Corollary \ref{ludwig_cor}. Let $\dualcat_G(\OO)_{\rbarss}$ be the full subcategory of $\dualcat_G(\OO)$, which is anti-equivalent to $\Mod^{\ladm}_G(\OO)_{\rbarss}$
 via Pontryagin duality. We refer the reader to \cite{brumer} for the basics on pseudocompact rings and completed tensor products. 

\begin{prop} \label{test_flatness}
Assume that $\rbarss=\chi_1\oplus \chi_2$ with $\chi_1\chi_2^{-1}\neq \omega^{\pm 1}$. Let $M\in \dualcat_G(\OO)_{\rbarss}$  and let $A$ be a pseudocompact ring together with a continuous action on $M$ via a ring homomorphism $A\rightarrow  \End_{\dualcat(\OO)}(M)$. Then for all right pseudocompact $A$-modules $\md$ we have a natural isomorphism: 
$$ \Shat^1(\md \wtimes_A M)\cong \md\wtimes_A\Shat^1(M).$$ 
In particular, if $M$ is a pro-flat $A$-module, in the sense  that the functor $\md \mapsto \md\wtimes_A M$  is exact, then so is $\Shat^1(M)$.
\end{prop}
\begin{proof} 
Since the functor $\Sc^1:  \Mod^\ladm_G(\OO)_{\rbarss} \longrightarrow \Mod^\ladm_{G_{\Qp} \times D^\times}(\OO)$ commutes with direct sums and is exact by Corollary \ref{ludwig_cor}, the functor $\Shat^1:\dualcat_G(\OO)_{\rbarss}\rightarrow \dualcat_{G_{\Qp} \times D^\times}(\OO)$ commutes with products and is also exact. Since any pseudocompact $A$-module $\md$ can be presented as 
$$\prod_{i\in I} A \rightarrow \prod_{j\in J} A \rightarrow \md \rightarrow 0,$$ 
for some sets $I$ and $J$,  the proposition is a formal consequence
of these two properties, see the proof of Proposition 2.4 in \cite{duke}, which is based on ideas of Kisin in \cite{kisin}.
\end{proof}
\begin{lem} Let $A$ be a complete local noetherian $\OO$-algebra with residue field $k$ and let $M$ be a 
pseudocompact $A$-module. Then $M$ is pro-flat if and only if it is flat.
\end{lem}
\begin{proof} Let $\widehat{\Tor}_i^A(\ast, M)$ be the $i$-th left derived functor of $\ast \wtimes_A M$. 
If $\md$ is a finitely generated $A$-module then, since $A$ is noetherian, $\md$ has a resolution by free 
$A$-modules of finite rank. Since $A^n \otimes_A M\cong M^n \cong A^n \wtimes_A M$, we 
conclude that $\widehat{\Tor}_i^A(\md, M)\cong \Tor_i^A(\md, M)$ for all finitely generated $\md$. If
$M$ is pro-flat then applying this observation to $\md=A/I$ for any ideal $I$ of $A$, we deduce that 
the map $I\otimes_A M\rightarrow M$ is injective and hence $M$ is flat. If $M$ is flat then taking $\md=k$
we obtain that $\widehat{\Tor}_1^A(k, M)=0$, and a standard application of the topological Nakayama's lemma shows that  $M$ is topologically free, and hence pro-flat, see Proposition 0.3.8 in Expos\'e $VII_B$ in SGA3. 
\end{proof}

\section{Fibres and flatness}\label{fandf}

In this section we explain a variation on \cite[Proposition A.30]{gee_newton}, where the authors prove a version of 
`miracle flatness' in a non-commutative setting. Let $\Lambda$ be an Auslander regular ring, we refer 
the reader to \cite[\S A.1]{gee_newton} or \cite{venjakob} for the definition. If $M$ is a finitely generated $\Lambda$-module 
then the grade of $M$ over $\Lambda$ is defined as 
$$j_\Lambda(M):=\inf\{ i: \Ext^i_\Lambda(M, \Lambda)\neq 0\}$$
and the dimension of $M$ over $\Lambda$ is defined as 
$$\delta_\Lambda(M):= \mathrm{gld}(\Lambda)- j_\Lambda(M),$$
where $\mathrm{gld}(\Lambda)$ is the global dimension of $\Lambda$. We say that $M$ is \textit{Cohen-Macaulay} if $\Ext^i_{\Lambda}(M, \Lambda)$ is non-zero for a single degree $i$. 
In particular, finite free modules are Cohen-Macaulay. 

Let $K$ be a compact torsion-free $p$-adic 
analytic pro-$p$ group. Then the rings $k[[K]]$, $\OO[[K]]$, $\OO[[K]]\otimes_\OO L$ are Auslander regular of 
global dimension $\dim K$, $1+\dim K$ and $\dim K$, respectively, where $\dim K$ is the dimension of
$K$ as a $p$-adic manifold. If $K$ is any compact $p$-adic analytic group  then the ring $\OO[[K]]$ might not be Auslander regular in general, but we may choose
an open pro-$p$ subgroup $K_1$ as above, and for a finitely generated $\OO[[K]]$-module $M$, define
$\delta_{\OO[[K]]}(M):=\delta_{\OO[[K_1]]}(M)$. Then $\delta_{\OO[[K]]}(M)$ does not depend on a choice of $K_1$ and has the expected properties of a dimension function. We will sometimes omit 
$\OO[[K]]$ from the notation, and just write $\delta(M)$.

\begin{prop} \label{use_GN}Let $A=\OO[[x_1,\ldots, x_r]]$ and let $M$ be an $A[[K]]$-module, which is finitely generated as an 
$\OO[[K]]$-module. Then
\begin{equation}\label{eqn_dim_bound}
\delta_{\OO[[K]]}(k\otimes_A M)+ \dim A \ge  \delta_{\OO[[K]]} (M).
\end{equation}
If $M$ is $A$-flat then 
\begin{equation}\label{eqn_dim}
\delta_{\OO[[K]]}(k\otimes_A M)+ \dim A= \delta_{\OO[[K]]} (M).
\end{equation}
If $M$ is Cohen--Macaulay as $\OO[[K]]$-module then \eqref{eqn_dim} implies that $M$ is $A$-flat. 
\end{prop}
\begin{proof} The first assertion follows from Lemma A.15 of \cite{gee_newton}, which says that 
if we mod out one relation then the dimension either goes down by one or stays the same.
The other assertions follow from Proposition A.30 in \cite{gee_newton} by observing that
$\mathrm{gld}(A[[K]])= \dim A  +\mathrm{gld}(k[[K]])$ and $\delta_{A[[K]]}(M)=\delta_{\OO[[K]]}(M)$, 
$\delta_{k[[K]]}(k\otimes_A M)= \delta_{\OO[[K]]}(k\otimes_A M)$
by \cite[Lemma A.19]{gee_newton}. Note that the ring $A[[K]]$ is again Auslander regular. 
\end{proof}

 \begin{cor}\label{find_A} Let $R$ be a complete local noetherian $\OO$-algebra with residue field $k$.
Let $M$ be an $R[[K]]$-module, which is finitely generated over $\OO[[K]]$. Assume that 
$M$ is $\OO$-torsion free, $M$ is Cohen-Macaulay over $\OO[[K]]$ and the action of $R$ on 
$M$ is faithful. Then 
\begin{equation}\label{eqn_dim1}
\delta_{\OO[[K]]}(k\otimes_R M)+ \dim R\ge \delta_{\OO[[K]]} (M).
\end{equation}
and equality holds if and only if there is a subring $A\subset R$, such that $R$ is a finite $A$-module, 
$A$ is formally smooth over $\OO$ and  $M$ is $A$-flat.
\end{cor}
\begin{proof} Since $R$ acts faithfully on $M$ and $M$ is $\OO$-torsion free, $R$ is $\OO$-torsion free. 
By Cohen's structure theorem for complete local rings there is a subring $A\subset R$ such that $R$ is finite 
over $A$ and $A\cong \OO[[x_1, \ldots, x_r]]$, see \cite[Theorem 29.4 (iii)]{mat} and the Remark 
following it.  Note that this implies that $\dim A=\dim R$.

Since $k\otimes_A M \cong (k \otimes_A R) \otimes_R M$, it admits $k\otimes_R M$ as a quotient, and 
since $k \otimes_A R$ is an $R$-module of finite length, it has a filtration of finite length with graded pieces isomorphic to subquotients of $k\otimes_R M$. Lemma
A.8 in \cite{gee_newton} implies that $\delta_{\OO[[K]]}(k\otimes_A M)= \delta_{\OO[[K]]}(k\otimes_R M)$. The corollary now follows from the proposition.
\end{proof}

\begin{lem}\label{dim_0} Let $M$ be a finitely generated $k[[K]]$-module. Then $\delta_{k[[K]]}(M)=0$ if and only if $M$ is a finite dimensional $k$-vector space. 
\end{lem}
\begin{proof} We may assume that $K$ is a uniform pro-$p$ group of dimension $d$. Let $\mm$ be the maximal ideal of $k[[K]]$, 
then the graded ring $\gr_{\mm}^\bullet (k[[K]])$ is a polynomial ring in $d$-variables, \cite[Thm.\,8.7.10]{wilson} and 
$\gr_{\mm}^\bullet(M )$ is a finitely generated $\gr_{\mm}^\bullet (k[[K]])$-module of dimension equal 
to $\delta_{k[[K]]}(M)$, \cite[Prop.\,5.4]{ardakov_brown}. In particular, the degree of the Hilbert polynomial of $\gr_\mm^{\bullet}(M)$ is equal to $\delta_{k[[K]]}(M)$. If $M$ is a finite dimensional $k$-vector space then $\delta_{k[[K]]}(M)=0$.
Conversely, if $\delta_{k[[K]]}(M)=0$ then the 
Hilbert polynomial of $\gr_\mm^{\bullet}(M)$ is constant and thus $\mm^j M= \mm^{j+1} M$ for some 
$j$. Nakayama's lemma implies that $\mm^j M=0$. Since $M$ 
is finitely generated over $k[[K]]$, we deduce that $M$ is a finite dimensional $k$-vector space. 
\end{proof}

\section{Local-global compatibility}\label{lg}
Let $p$ be a prime and let $F$ be a totally real number field with a fixed place $\pp$ above $p$, such that $F_{\pp}=\Qp$,  and a fixed infinite place $\infty_F$.  
Let $D_0$ be a quaternion algebra with centre $F$, ramified at all the infinite places of $F$ and split at $\pp$. Let $\Sigma$ be a set of finite ramification places of $D_0$. We fix a maximal order $\OO_{D_0}$ of $D_0$, and for each finite place $v\not\in \Sigma$ an isomorphism $(\OO_{D_0})_v\cong M_2(\OO_{F_v})$. For each finite place $v$ of $F$ we will denote by 
$\mathbf N(v)$ the order of the residue field at $v$, and by $\varpi_v\in F_v$ a uniformizer. 

Denote by $\AfF\subset \mathbb A_F$ the finite adeles and the adeles respectively. 
Let $U=\prod_v U_v$ be a compact open subgroup contained in $\prod_v (\OO_{D_0})_v^{\times}$. We 
assume that $U_\pp=\GL_2(\Zp)=K$ and that $U_v$ is a pro-$p$ group at other places above $p$. We may write 
\begin{equation}\label{double_coset}
(D_0 \otimes_F \AfF )^{\times} = \coprod_{i\in I} D_0^{\times}t_i U(\AfF )^{\times}
\end{equation}
 for some $t_i \in  (D_0 \otimes_F \AfF )^{\times}$ and a finite index set $I$, where we have identified $(\AfF )^{\times}$ with the centre 
 of $(D_0 \otimes_F \AfF )^{\times}$. In the arguments that follow we are free to replace $U$ by a smaller open subgroup by shrinking $U_v$ at any place $v$ different from $\pp$.
 In particular, we may assume that 
 \begin{equation}\label{trivial}
 (U (\AfF)^{\times}\cap t_i^{-1}D_0^{\times} t_i)/F^{\times} =1 
 \end{equation}
 for all $i\in I$, see for example Lemma 3.2 in \cite{blocksp2}.  We will use standard notation for subgroups of $U$, so for example $U^{\pp}_p= \prod_{v\mid p, v\neq \pp} U_v$, $U^p= \prod_{v\nmid p} U_v$.
 
 If $A$ is a topological $\OO$-algebra we let $S(U^p, A)$ be be the space of continuous functions 
 $$ f: D_0^{\times}\backslash (D_0\otimes_F \AfF)^{\times}/U^p\rightarrow A.$$
   The group  $(D_0\otimes\Qp)^\times$ acts continuously on $S(U^p, A)$ by right translations. 
 It follows from \eqref{trivial} that the map $f\mapsto [u \mapsto \sum_{i\in I} f(t_i u)]$ induces an 
 isomorphism of $U(\AfF)^{\times}$-representations
 \begin{equation}\label{beijing}
 S(U^p, A)\overset{\cong}{\longrightarrow} \bigoplus_{i\in I} C(U^p  F^{\times}\backslash  U (\AfF)^{\times}, A),
 \end{equation}
 where $C$ denotes the space of continuous functions. Let $\psi: (\AfF )^{\times}/F^{\times}\rightarrow \OO^{\times}$ be a continuous character  such that $\psi$ is trivial of  
 $(\AfF )^{\times}\cap U^p$. We may consider $\psi$ as an $A$-valued character, via $\OO^{\times}\rightarrow A^\times$. Let 
 $S_\psi(U^p, A)$ be the $A$-submodule of $S(U^p, A)$, consisting of functions 
 such that $f(gz)=\psi(z) f(g)$ for all $z\in (\AfF)^{\times}$. The isomorphism  \eqref{beijing} 
 induces an isomorphism of $U_p$-representations:
 \begin{equation}\label{beijing_1}
 S_{\psi}(U^p, A)\overset{\cong}{\longrightarrow} \bigoplus_{i\in I} C_\psi(U_p, A),
 \end{equation}
where $C_{\psi}$ denotes the continuous functions on which the centre acts by the character $\psi$.
 One may think of $S_\psi(U^p, A)$ as the space of  algebraic automorphic forms on $D_0^\times$ with tame level $U^p$ and no restrictions on the
 level or weight at places dividing $p$. We want to introduce a variant by fixing the level and weight 
 at places dividing $p$, different from $\pp$. Let $\lambda$ be a continuous representation
 $U^{\pp}_p$ on a free $\OO$-module of finite rank, such that $(\AfF)^\times\cap U^{\pp}_p$ acts 
 on $\lambda$ by the restriction of $\psi$ to this group. We let 
 $$S_{\psi, \lambda}(U^{\pp}, A):= \Hom_{U^\pp_p}(\lambda, S_\psi(U^p, A)).$$
  We will omit $\lambda$ as an index if it is the trivial representation. Let us note that a presentation 
  $\OO[[U_p^{\pp}]]^{\oplus n}\rightarrow \OO[[U_p^{\pp}]]^{\oplus m}\twoheadrightarrow \lambda$ gives us an exact sequence: 
  \begin{equation}\label{beijing_3}
  0\rightarrow S_{\psi, \lambda}(U^\pp, A)\rightarrow S_{\psi}(U^p, A)^{\oplus m}\rightarrow  S_{\psi}(U^p, A)^{\oplus n}.
  \end{equation}
  If $A$ is an $\OO/\varpi^n$-module then there is an open subgroup $V_p^{\pp}$ of  $U^{\pp}_p$, which acts trivially on $\lambda/ \varpi^n$.
  If we let $V^{\pp}:= U_p V^\pp_p$ then by  taking $V_p^{\pp}$-invariants of \eqref{beijing_3} we have an exact sequence 
    \begin{equation}\label{beijing_4}
  0\rightarrow S_{\psi, \lambda}(U^\pp, A)\rightarrow S_{\psi}(V^{\pp}, A)^{\oplus m}\rightarrow  S_{\psi}(V^{\pp}, A)^{\oplus n}.
  \end{equation}
If the topology on $A$ is discrete, for example if $A=L/\OO$ or $A=\OO/\varpi^n$ then we have 
    \begin{equation}\label{beijing_5}
    S_{\psi}(U^p, A)\cong \varinjlim_{U_p^{\pp}} S_{\psi}(U^\pp, A).
    \end{equation}
  The action of 
 $(D_0\otimes\Qp)^\times$ on $S_\psi(U^p, A)$ by right translations induces 
 a continuous  action of $(D_0\otimes_F  F_{\pp})^{\times}\cong \GL_2(\Qp)=G$ on $S_\psi(U^p, A)$ and 
 $S_{\psi, \lambda}(U^{\pp}, A)$. Let $\zeta:\Qp^{\times}\rightarrow \OO^\times$ be the character obtained by restricting $\psi$ to $F_{\pp}^\times$.

 \begin{lem}\label{admissible} The representations $S_\psi(U^p, L/\OO)$, $S_{\psi, \lambda}(U^\pp, L/\OO)$ lie in $\Mod^{\ladm}_{G, \zeta}(\OO)$. Moreover, $S_{\psi, \lambda}(U^\pp, L/\OO)$ is admissible.
\end{lem}
\begin{proof} The first assertion  follows from \eqref{beijing_1}, which also implies that we have 
an isomorphism of $K$-representations
\begin{equation}\label{beijing_2}
S_{\psi, \lambda}(U^{\pp}, L/\OO)\cong \bigoplus_{i\in I} \lambda^{\vee} \otimes_{\OO} C_{\zeta}(K, L/\OO), 
\end{equation}
where $\lambda^{\vee}:=\Hom^{\cont}_{\OO}(\lambda, L/\OO)$ denotes the Pontryagin dual of 
$\lambda$. Since the set $I$ is finite and $\lambda$ is a finite $\OO$-module, we deduce the second assertion. 
\end{proof}

\begin{lem}\label{injective} The restrictions of $S_\psi(U^p, L/\OO)$  and $S_{\psi, \lambda}(U^\pp, L/\OO)$ to $K$ are injective in $\Mod^{\sm}_{K, \zeta}(\OO)$.
\end{lem}
\begin{proof} This follows from \eqref{beijing_1} and \eqref{beijing_2}.
\end{proof}

Let $S$ be  a finite set of places of $F$ containing $\Sigma$, all the places above 
$p$, all the infinite places and all the places $v$, where $U_v$ is not maximal and all the ramification places of $\psi$. Let $\TT^{\univ}_{S}=\OO[ T_v, S_v]_{v\not\in S}$ 
be a commutative polynomial ring in the indicated formal variables. If $A$ is a topological $\OO$-algebra
then
$S_{\psi}(U^p, A)$ and $S_{\psi, \lambda}(U^\pp, A)$ become $\TT^{\univ}_{S}$-modules with $S_v$ acting via the double coset 
$U_v \bigl ( \begin{smallmatrix} \varpi_v & 0 \\ 0 & \varpi_v\end{smallmatrix}\bigr) U_v$ and $T_v$ acting via the double coset 
$U_v \bigl ( \begin{smallmatrix} \varpi_v & 0 \\ 0 & 1\end{smallmatrix}\bigr) U_v$. 

Let $G_{F,S}$ be 
the absolute Galois group of the maximal extension of $F$ in $\overline{F}$ which is unramified outside $S$. Let $\rhobar: G_{F, S}\rightarrow \GL_2(k)$ be a continuous absolutely  irreducible representation.
Let $\mathfrak{m}$ be the maximal ideal of $\mathbb{T}_{S}^{\univ}$ generated  by $\varpi$ and all elements, which reduce modulo 
$\varpi$ to  $T_v- \tr \rhobar(\Frob_v), S_v \mathbf N(v)- \det \rhobar(\Frob_v)$ for all $v\not\in S$. 
We denote by $S_{\psi}(U^p, L/\OO)_\mm$ and $S_{\psi, \lambda}(U^\pp, L/\OO)_\mm$ the localisations 
at $\mm$, and assume that they are non-zero for fixed $\psi$ and $\lambda$.

\begin{lem}\label{K-injective} The representations $S_{\psi}(U^p, L/\OO)_\mm$ and  $S_{\psi, \lambda}(U^\pp, L/\OO)_\mm$ are
 direct summands of $S_{\psi}(U^p, L/\OO)$ and  $S_{\psi, \lambda}(U^\pp, L/\OO)$, respectively. 
 In particular, their restrictions to $K$ are injective in $\Mod^{\sm}_{K, \zeta}(\OO)$. Moreover, if $S(U^p, L/\OO)_{\mm}$ (resp. $S_{\lambda}(U^\pp, L/\OO)_{\mm}$) 
 is non-zero then $S_{\psi}(U^p, L/\OO)_{\mm}$ (resp. $S_{\psi, \lambda}(U^\pp, L/\OO)_{\mm}$) is non-zero if and only if $\det \rhobar \equiv \psi \chi_{\cyc} \pmod{\varpi}$.
 \end{lem}
 \begin{proof} If $U'_p$ is an open subgroup of $U_p$ then $S_{\psi}(U^p, \OO/\varpi^n)^{U_p'}$ 
 is a finitely generated $\OO/\varpi^n$-module and Chinese remainder theorem implies that
 $(S_{\psi}(U^p, \OO/\varpi^n)^{U_p'})_\mm$ is a direct summand of $S_{\psi}(U^p, \OO/\varpi^n)^{U_p'}$. By passing to a direct limit 
 over all such $U_p'$ and $n\ge 1$ we obtain the assertion for  $S_{\psi}(U^p, L/\OO)$. The 
 argument for $S_{\psi, \lambda}(U^\pp, L/\OO)$ is the same. The second assertion follows from Lemma \ref{injective}. Since $S_{\psi}(U^p, L/\OO)$ is a
 a union of $\mathbb{T}_{S}^{\univ}$-submodules, which are $\OO$-modules of finite length, its localisation at $\mm$ is non-zero if and only if 
 the $\mm$-torsion subspace $S_{\psi}(U^p, L/\OO)[\mm]$ is non-zero. The same argument applied to $S(U^p, L/\OO)$ shows that
 $S(U^p, L/\OO)[\mm]$ is non-zero. Since the Hecke operators $S_v$ act on $S(U^p, L/\OO)[\mm]$ with eigenvalue $\det \rhobar(\Frob_v)  \mathbf N(v)^{-1}$
 for all $v\not\in S$,  we deduce from Chebotarev's density theorem that $(\AfF )^{\times}$ acts on $S(U^p, L/\OO)[\mm]$ by the character $\chi_{\cyc} \det \rhobar$. 
 Hence, $S_{\psi}(U^p, L/\OO)[\mm]$ is non-zero if and only if $\det \rhobar \equiv \psi \chi_{\cyc} \pmod{\varpi}$. The same argument applies to $S_{\psi, \lambda}(U^\pp, L/\OO)_{\mm}$.
 \end{proof}
 
 We let 
$$ S_{\psi}(U^p, \OO)_\mm:= \varprojlim_n S_{\psi}(U^p, \OO/\varpi^n)_\mm, \quad 
S_{\psi, \lambda}(U^\pp, \OO)_\mm:= \varprojlim_n S_{\psi, \lambda}(U^\pp, \OO/\varpi^n)_\mm$$
equipped with the $p$-adic topology. It follows from \eqref{beijing_1} that for all $n\ge 1$ the map
$S_{\psi}(U^p, \OO)_\mm/\varpi^n \rightarrow S_{\psi}(U^p, \OO/\varpi^n)_\mm$ is an isomorphism, 
    and \eqref{beijing_2} implies that the same holds for $S_{\psi, \lambda}(U^\pp, \OO)_{\mm}$.
It follow from the discussion in section \ref{dualcat} that we have natural homeomorphisms
\begin{equation}\label{more_duals}
S_{\psi}(U^p, \OO)_\mm\cong (( S_{\psi}(U^p, L/\OO)_\mm)^{\vee})^d, \quad S_{\psi, \lambda}(U^\pp, \OO)_\mm\cong (( S_{\psi,\lambda}(U^\pp, L/\OO)_\mm)^{\vee})^d.
\end{equation}

Let $\rbar$ denote the restriction of $\rhobar$ to the decomposition group at $\pp$, which we identify with 
the absolute Galois group of $\Qp$. 

\begin{prop}\label{right_object} $S_\psi(U^p, L/\OO)_\mm$ and  $S_{\psi, \lambda}(U^\pp, L/\OO)_\mm$ lie in $\Mod^{\ladm}_{G, \zeta}(\OO)_{\rbarss}$.
\end{prop}
\begin{proof}  We first observe that the centre of $G$ acts on both $S_\psi(U^p, L/\OO)_\mm$ and  
$S_{\psi, \lambda}(U^\pp, L/\OO)_\mm$ by the restriction of $\psi$ to $F_{\pp}^{\times}$, which is 
equal to $\zeta$ by definition.

Since for any $\pi\in \Mod^{\ladm}_G(\OO)$, $\soc_G \pi \hookrightarrow \pi$ is essential, where
$\soc_G$ denotes the maximal semi-simple subrepresentation, 
it follows from \eqref{decompose_can} that it is enough to show that all irreducible subrepresentations of 
 $S_{\psi, \lambda}(U^\pp, L/\OO)_\mm$  and $S_{\psi}(U^p, L/\OO)_\mm$ lie in $\BB_{\rbarss}$. 
 Since any such representation is killed by $\varpi$ we may replace $L/\OO$ with $k$.
 
 It is enough to prove the statement for $S_\psi(U^{\pp}, k)_\mm$ with 
 $U_p^\pp$ arbitrary small, since then \eqref{beijing_4} and \eqref{beijing_5} imply the assertion in general.

Let $\pi$ be an irreducible subrepresentation of $S_\psi(U^\pp, k)_\mm$.   After enlarging $L$ we may assume that $\pi$ is absolutely irreducible. Let $\sigma$ be an irreducible $K$-subrepresentation of $\pi$. Then $\sigmabar$ is isomorphic to 
$\Sym^b k^2 \otimes \det^a$, for uniquely determined integers $0\le b\le p-1$, $0\le a\le p-2$. It follows from \cite{bl} that 
$\End_G(\cInd_K^G \sigma)\cong k[T, S^{\pm 1}]$, where the Hecke operators  $T$ and $S$ correspond to the double cosets
$K \bigl ( \begin{smallmatrix} p & 0 \\ 0 & 1 \end{smallmatrix} \bigr) K$ and $K \bigl ( \begin{smallmatrix} p & 0 \\ 0 & p \end{smallmatrix} \bigr) K$, respectively. Moreover, there are $\lambda, \mu\in k$ such that $\pi$ is a quotient of 
$$ \cInd_K^G \sigma/ ( T-\lambda, S-\mu).$$
We claim that one may read off $\lambda$, $\mu$ and the possible values of $(a, b)$ from the restriction of $\rhobar$ to
the decomposition group at $\pp$.  It then follows from \cite{bl} and \cite{breuil1}, which describe the irreducible quotients 
of $\cInd_K^G \sigma/ ( T-\lambda, S-\mu)$, that $\pi$ lies in~$\BB_{\rbarss}$. These arguments are by now fairly standard and appear in the weight part of Serre's conjecture, so we only give a sketch. 

We first modify the setting slightly: if $\psi': (\AfF )^{\times}/F^{\times}\rightarrow \OO^{\times}$ is a character congruent to $\psi$ modulo $\varpi$ and $\lambda'$ is a representation of $U^\pp_p$ on 
a finite $\OO$-module with central character $\psi'$ and $U^{\pp}_p$ is a pro-$p$ group, then 
 the $U^{\pp}_p$-invariants of $\lambda'/\varpi$ are  non-zero, and thus $S_\psi(U^{\pp}, k)_\mm$
 is a $G$-invariant subspace of $S_{\psi', \lambda'}(U^{\pp}, k)_\mm$.  
 We may choose $\psi'= \chi_{\cyc}^{b+2a}\alpha$, where $a$ and $b$ are as above and $\alpha$ is the Teichm\"uller lift of 
 the character $\psi \chi_{\cyc}^{-b -2a} \pmod{\varpi}$ and $\chi_{\cyc}$ is the $p$-adic cyclotomic character and $\lambda'=\otimes_{\iota} (\Sym^b \OO^2\otimes \det^a)$, where the tensor product 
 is taken over all embeddings $\iota: F\hookrightarrow L$, which do not factor through $F_\pp$.
 Note that since $U^\pp_p$ is assumed to be pro-$p$ the character $\psi \chi_{\cyc}^{-b -2a} \pmod{\varpi}$
is trivial on $U^\pp_p \cap (\AfF)^\times$, which implies that $\alpha$ is trivial on $U^\pp_p \cap (\AfF)^\times$, and thus $\lambda'$ has central character $\psi'$. Note that since $\zeta$ is also the central character of $\sigmabar$, we have  $\zeta(x)\equiv x^{b+2a}\pmod{\varpi}$, for all $x\in \Zp^{\times}$. To ease the notation we drop the superscript  $'$ from 
$\psi'$ and $\lambda'$.

Since $S_{\psi, \lambda}(U^\pp, k)_\mm$ is an admissible representation by Lemma \ref{admissible}, the $k$-vector space 
$$\Hom_G(\pi, S_{\psi, \lambda}(U^\pp, k)_\mm)$$ is finite dimensional. Thus we may assume that 
$\pi$ is contained in $S_{\psi, \lambda}(U^\pp, k)[\mm]$. Let $$\sigma^\circ:=\Sym^b \OO^2 \otimes \dt^a, \quad \sigma:=\Sym^b L^2 \otimes \dt^a.$$  Since
$S_{\psi, \lambda}(U^\pp, L/\OO)_\mm$ is  admissible and injective in $\Mod^{\sm}_{K, \zeta}(\OO)$ by Lemmas \ref{admissible}, \ref{injective}, using \eqref{more_duals} we see that 
$\Hom_K(\sigma^{\circ}, S_{\psi, \lambda}(U^\pp, \OO)_{\mm})$ is a free $\OO$-module of finite rank, which is  congruent to $\Hom_K(\sigmabar, S_{\psi, \lambda}(U^\pp, k)_\mm)$ modulo $\varpi$. It follows from \cite[Lemme 6.11]{DS} that after replacing $L$ with a finite extension 
$$\Hom_K(\sigma^{\circ}, S_{\psi, \lambda}(U^\pp, \OO)_{\mm})\otimes_\OO L$$ contains an eigenvector $\phi$ for 
all the Hecke operators in $\TT^{\univ}_S$ with eigenvalues lifting those given by $\mm$. By evaluating $\phi$ we obtain an automorphic form $f$ on $D_0^\times$ such that the associated Galois representation $\rho_f$ lifts $\rhobar$ and its restriction 
to the decomposition group at $\pp$ is crystalline with Hodge--Tate weights $(1-a, -a-b)$, where we adopt the conventions
of \cite{6auth2}, so that the cyclotomic character has Hodge--Tate weight equal to $-1$. We note that the difference of the 
two Hodge--Tate weights is equal to $b+1$, which is between $1$ and $p$.  In particular, $\sigmabar$ is a Serre weight for $\rbar$. The possibilities for the pair $(a,b)$ are listed in the proof of \cite[Lemma 2.15]{6auth2}. The compatibility of local and global Langlands correspondence implies that the
$G$-subrepresentation of $ S_{\psi, \lambda}(U^\pp, \OO)_{\mm}\otimes_\OO L$ generated by the image of $\phi$ is of the form 
$\Psi \otimes \sigma$, where $\Psi$ is a smooth unramified principal series representation. Moreover, one may read off 
the Satake parameters of $\Psi$ from the Weil--Deligne representation associated to $\rho_f|_{G_{F_{\pp}}}$, see \cite[Proposition 2.9]{6auth2}. It follows from \cite{breuil2} that $\lambda$ and $\mu$ are reductions modulo $\varpi$ of 
the Satake parameters of $\Psi$, rescaled by a suitable power of $p$, see the proof of  \cite[Lemma 2.15]{6auth2}.
  It then follows from \cite{bl} and \cite{breuil1}, which describe the irreducible quotients 
of $\cInd_K^G \sigma/ ( T-\lambda, S-\mu)$, that $\pi$ lies in $\BB_{\rbarss}$. One could sum up the last part of the argument as the compatibility of $p$-adic and mod-$p$ Langlands correspondences. 
\end{proof}

Let $R^{\ps}_{\tr\rbar}$ be the universal deformation ring parameterizing $2$-dimensional pseudo-characters of $G_{\Qp}$ lifting $\tr \rbar$.
We have two actions of $R^{\ps}_{\tr\rbar}$ on $S_\psi(U^\pp, L/\OO)_\mm$. The first action is given by the 
composition 
$$\theta_1:R^{\ps}_{\tr \rbar} \rightarrow R^{\univ}_{\rhobar} \rightarrow \T(U^\pp)_\mm\rightarrow \End_G^{\cont}(S_{\psi}(U^{\pp}, \OO)_\mm),$$
where  $R^{\univ}_{\rhobar}$ is the universal global deformation ring of $\rhobar$, and the first arrow is obtained by considering the restriction of the trace of the universal deformation of $\rhobar$ to $G_{F_\pp}$. The second arrow is obtained by associating Galois representations to automorphic forms, see \cite[Proposition 5.7]{scholze}. The third arrow is given by \cite[Corollary 7.3]{scholze}. The second action 
$$ \theta_2: R^{\ps}_{\tr \rbar} \rightarrow \End_G^{\cont}(S_{\psi}(U^{\pp}, \OO)_\mm),$$
is given by interpreting $R^{\ps}_{\tr \rbar}$ as the centre of the category $\Mod^{\ladm}_G(\OO)_{\rbarss}$ using \cite{image}.  In order to apply the results of \cite{image},  if $\rbarss= \chi \oplus \chi \omega$ for some character $\chi:G_{\Qp}\rightarrow k^\times$ then we assume that $p\ge 5$. In \cite{image} we work with a fixed central character, but 
this can be unfixed by using the ideas of \cite[\S 6.5, Cor.\,6.23]{6auth2}.
Since the centre of the category acts naturally on every object in the category, Proposition \ref{right_object} implies that $R^{\ps}_{\tr \rbar}$ acts naturally on $S_\psi(U^\pp, L/\OO)_\mm$  and using \eqref{more_duals} we obtain a natural action of $R^{\ps}_{\tr \rbar}$ on $S_{\psi}(U^{\pp}, \OO)_\mm$, which we denote by  $\theta_2$. 

\begin{prop}\label{equal_actions} Assume that $\psi$ is of finite, prime to $p$ order and $\psi$ is trivial on $U^\pp\cap (\AfF)^\times$. Then the two actions $\theta_1$ and $\theta_2$ of $R^{\ps}_{\tr \rbar}$ on 
$S_{\psi}(U^{\pp}, L/\OO)_\mm$ coincide.
\end{prop} 
\begin{proof} The result follows from three ingredients: the theory of capture \cite[\S 2.4]{CDP},
\cite[\S 2.1]{blocksp2}, which is based on the ideas of Emerton in \cite{emerton_lg}, the results of Berger--Breuil \cite{bb} on the universal unitary completions of locally algebraic principal series representations and the results of Colmez in \cite{colmez} on the compatibility of the $p$-adic and classical local Langlands correspondence, as we will now explain. 

It follows from \eqref{more_duals} that it is enough to show that the two actions coincide on 
$S_\psi(U^{\pp}, \OO)_\mm$.
Let $\Pi(U^{\pp})$ be the $L$-Banach space representation of $G$ defined by
$$\Pi(U^\pp):=S_\psi(U^{\pp}, \OO)_\mm\otimes_{\OO} L.$$
 Since 
$S_\psi(U^{\pp}, \OO)_\mm$ is $\OO$-torsion free it is enough to show that $\theta_1$ and $\theta_2$ agree on 
$\Pi(U^{\pp})$. Since the Schikhof dual of $S_{\psi}(U^{\pp}, \OO)_\mm$ is isomorphic 
to $S_{\psi}(U^\pp, L/\OO)_{\mm}^{\vee}$, Lemmas \ref{injective} and \ref{admissible} imply that  
$S_{\psi}(U^{\pp}, \OO)_\mm^d$ is 
projective in the category $\Mod^{\pro}_{K, \zeta}(\OO)$ and is finitely generated over $\OO[[K]]$.

Let $\eta:\Zp^\times\rightarrow \OO^\times$ be the non-trivial character of order $2$. Let 
$\alpha: \Zp^\times\rightarrow \OO^\times$ be a character of order $p$. Let $J:=\bigl (\begin{smallmatrix} 
\Zp^{\times} & \Zp \\ p^2 \Zp & \Zp^{\times}\end{smallmatrix}\bigr )$ and let $\chi$ be a character of $J$ defined by 
 $$\chi(\bigl (\begin{smallmatrix} a& b \\  c & d\end{smallmatrix}\bigr )):=\zeta(a)\alpha(a) \alpha(d)^{-1}.$$ 
Let $\tau:= \Ind_J^K \chi$, and let $V_a:= \tau \otimes_L \Sym^{2a} L^2 \otimes \det^{-a}$, $a\ge 0$. 
Note that the central character of $V_a$ is equal to $\zeta$. 
It follows from the proof of Proposition 2.7 in \cite{blocksp2} that if $\varphi$ is a continuous $K$-equivariant endomorphism of $\Pi(U^{\pp})$, which kills 
$$\Hom_K(V_a, \Pi(U^{\pp})) \quad \text{and}\quad \Hom_K(V_a\otimes \eta\circ \det, \Pi(U^\pp))$$
for all $a\ge 0$,  then $\varphi$ is zero. Thus it is enough to show that the two actions of 
$R^{\ps}_{\tr \rbar}$ on the above modules coincide, so that $\theta_1(r)-\theta_2(r)$ 
annihilates them for all $r\in R^{\ps}_{\tr \rbar}$ and all $a\ge 0$. 

In the following let $V= V_{\sm} \otimes V_{\alg}$, 
where $V_{\alg}= \Sym^{2a} L^2\otimes \det^{-a}$, and $V_{\sm}= \tau$ or $\tau\otimes \eta\circ \det$.
Since $V$ is a locally algebraic representation of $K$, we have $\Hom_K (V, \Pi(U^{\pp}))= \Hom_K(V, \Pi(U^{\pp})^{\lalg})$, 
where $\Pi(U^\pp)^{\lalg}$ is the subspace of locally algebraic vectors in $\Pi(U^\pp)$.
This subspace can be identified with a subspace of classical automorphic forms on $D_0^{\times}$, see \cite[Lemma 1.3]{taylor_deg2}, \cite[\S 3.1.14]{kisin_annals}, \cite[\S 3]{interpolate}. In particular, the action of $\T(U^p)_\mm[1/p]$ on $\Pi(U^\pp)^{\lalg}$, and hence on  $\Hom_K(V, \Pi(U^\pp))$, is semi-simple. Since $S_{\psi}(U^{\pp}, \OO)^d_\mm$
is finitely generated as an $\OO[[K]]$-module, the vector space   $\Hom_K(V, \Pi(U^\pp))$ is finite dimensional. Thus for a fixed $V$, after replacing $L$ by a finite extension, we may assume that 
$\Hom_K(V, \Pi(U^\pp))$ has a basis of eigenvectors for the action of $\T(U^\pp)_\mm[1/p]$. 

Let $\phi\in\Hom_K(V, \Pi(U^\pp))$ be such an eigenvector. Then it is enough to show that $\phi$ is an eigenvector
 for the action $R^{\ps}_{\tr \rbar}$ via $\theta_2$ and that the annihilators of $\phi$ for the  two actions coincide, since then
 $\theta_1(r)-\theta_2(r)$ will kill $\phi$ for all $r\in R^{\ps}_{\tr \rbar}$.
 
Let $x\in \mSpec R^{\ps}_{\tr \rbar}[1/p]$ be the kernel for the action $R^{\ps}_{\tr \rbar}[1/p]$ on $\phi$ via $\theta_1$. Then by unravelling the definition 
of $\theta_1$ we get that $x$ corresponds to the pseudo-character $\tr (\rho_f |_{G_{F_\pp}})$, where $\rho_f$ is the Galois 
representation  attached to the automorphic form $f$, corresponding to the Hecke eigenvalues given by the action of 
$\T(U^p)_\mm[1/p]$ on $\phi$.  Moreover, by the same argument as in the proof of Proposition \ref{right_object},  $\rho_f |_{G_{F_\pp}}$, is potentially semistable with Hodge--Tate weights $(1+a, -a)$. As in the proof of Proposition \ref{right_object}
the $G$-subrepresentation of $\Pi(U^\pp)$ generated by the image of $\phi$ is of the form $\Psi\otimes V_{\alg}$, where $\Psi$ is an irreducible 
smooth  representation with $\Hom_K(V_{\sm}, \Psi)\neq 0$. The theory of types, see \cite{henniart}, implies 
that $\Psi\cong(\Ind_B^G \psi_1\otimes \psi_2 | \cdot |^{-1})_{\sm}$, where
$\psi_1, \psi_2:\Qp^\times \rightarrow L^\times$ are smooth characters, such that if $V_{\sm}=\tau$ then 
$\psi_1|_{\Zp^\times}=\zeta \alpha$, $\psi_2|_{\Zp^{\times}}= \alpha^{-1}$
and if $V_{\sm}=\tau\otimes \eta\circ \det$ then $\psi_1|_{\Zp^\times}=\zeta \alpha\eta$, $\psi_2|_{\Zp^{\times}}= \alpha^{-1}\eta$. Hence, $\rho_f |_{G_{F_\pp}}$ is potentially crystalline.
Moreover, $\Psi$ determines the Weil--Deligne representation 
associated to  $\rho_f |_{G_{F_\pp}}$ via the classical local Langlands correspondence.

The universal unitary completion of $\Psi \otimes V_{\alg}$ is absolutely irreducible, 
 by \cite[5.3.4]{bb}, \cite[2.2.1]{be}, and will coincide with the closure 
of $\Psi\otimes V_\alg$ in $\Pi(U^{\pp})$, which we denote by $\Pi$. The action of 
$R^{\ps}_{\tr \rbar}[1/p]$ on $\Pi(U^\pp)$ via $\theta_2$ preserves $\Pi$, since it acts as the center of the category. Schur's lemma implies that the annihilator of $\Pi$ is a maximal ideal of  $R^{\ps}_{\tr \rbar}[1/p]$, which we denote by $y$. Note that we have shown that the action of $R^{\ps}_{\tr \rbar}[1/p]$ via $\theta_2$ preserves $\phi$ and 
the annihilator is equal to $y$. So we have to show that $x=y$.

If $\Pi$ is non-ordinary then $r:=\cV(\Pi)$ is 
an absolutely irreducible  two dimensional potentially crystalline representation of $G_{\Qp}$ lifting $\rbar$ and it follows from 
\cite[Proposition 11.3]{image} that $y$ corresponds to $\tr r$. The compatibility of $p$-adic and 
classical Langlands correspondences, proved by Colmez in  \cite{colmez},  implies that $r$ and $\rho_f |_{G_{F_\pp}}$ have the same Weil--Deligne representations. Moreover, the Hodge--Tate weights of $r$ are determined by $V_{\alg}$ and are equal to those of $\rho_f$.
Since the Hodge--Tate weights of $r$ and $\rho_f |_{G_{F_\pp}}$ are equal, it follows from \cite[\S 4.5]{colmez_tri} that 
$r\cong \rho_f|_{G_{F_\pp}}$ and and so $x=y$. 

If $\Pi$ is ordinary then $\cV(\Pi)$ is one dimensional, and we denote the character by $\delta$. It follows from Corollaries 8.15, 9.37, Proposition 10.107 in 
\cite{image} that $y$ corresponds to the pseudo-character $\delta \oplus \delta^{-1} \zeta \varepsilon^{-1}$. Since $\Psi$ is irreducible we may assume that  $\val(\psi_1(p))\ge \val(\psi_2(p))$, since interchanging the characters will give an isomorphic representation. Then it follows from \cite[Lemma 12.5]{image} that $\Pi$ is isomorphic to a parabolic induction of a unitary character 
$\gamma_1\otimes\gamma_2$, such that $\gamma_1(x)= \psi_1(x) x^{-a}$, $\gamma_2(x)= \psi_2(x) |x|^{-1} x^{a}$ for all $x\in \Qp^{\times}$. It follows from the definition of $\cV$  that $\cV(\Pi)\cong \gamma_1 \varepsilon^{-1}$, so that 
$\delta \oplus \delta^{-1} \zeta \varepsilon^{-1}= \gamma_1 \varepsilon^{-1} \oplus \gamma_2$, where for the ease of 
notation  we use local class field theory to identify characters of $G_{\Qp}$ with unitary characters of $\Qp^\times$.  Note that to match the conventions of \cite{6auth2} we twist the definition of 
$\cV$ in \cite[\S 5.7]{image} by the inverse of the cyclotomic character. The cyclotomic character $\varepsilon$ corresponds to the character $x\mapsto x |x|$ via local class field theory. 
We have that $\gamma_1\varepsilon^{-1}=\psi_1|\cdot|^{a} \varepsilon^{-a-1}$ is crystalline with Hodge--Tate weight $a+1$ and Weil--Deligne representation
$\psi_1 |\cdot|^{-1}$, simirlarly $\gamma_2$ is crystalline with Hodge--Tate weight $-a$ and Weil--Deligne 
representation equal to $\psi_2|\cdot|^{-1}$. The Weil--Deligne representation attached to $\rho_f$ is the Weil--Deligne  
representation attached to $\Psi$ by the classical Langlands correspondence, which is equal to $\psi_1|\cdot|^{-1}\oplus 
\psi_2|\cdot|^{-1}$, see the calculation in the proof of Proposition 2.9 in \cite{6auth2}. It follows from 
 \cite[\S 4.5]{colmez_tri} that $r$ and $\rho_f|_{G_{F_\pp}}$ have the same semi-simplification. Hence 
 $x=y$.
\end{proof}

\begin{cor}\label{actions_coincide} The two actions  of $R^{\ps}_{\tr \rbar}$ on 
$S_{\psi}(U^{p}, L/\OO)_\mm$ and $S_{\psi, \lambda}(U^\pp, L/\OO)_\mm$ via $\theta_1$ and $\theta_2$ coincide.
\end{cor}
\begin{proof} If $\chi: (\AfF)^\times/F^{\times}\rightarrow \OO^{\times}$ is a continuous character
trivial on $U^p\cap (\AfF)^\times$ and trivial modulo $\varpi$ then the map $f\mapsto [g\mapsto f(g) \chi(\det (g))]$ induces an isomorphism of $\TT(U^p)_{\mm}[G]$-modules
$$ S_{\psi}(U^p, L/\OO)\otimes \chi\circ \det \overset{\cong}{\longrightarrow} S_{\psi\chi^2}(U^p, L/\OO).$$
This  induces an isomorphism between the $G$-endomorphism rings
$$\varphi:\End_G( S_{\psi}(U^p, L/\OO))\overset{\cong}{\longrightarrow} \End_G( S_{\psi\chi^2}(U^p, L/\OO)).$$ 
Moreover, for $i=1, 2$ we have $\phi\circ \theta_i= \theta_i\circ \mathrm{tw}_{\chi}$, where
$ \mathrm{tw}_{\chi}$ is the automorphism of   $R^{\ps}_{\tr \rbar}$ obtained by sending a deformation 
to its twist by $\chi$.  Thus if we prove the assertion for $\psi$ then we may deduce the assertion for
$\psi\chi^2$ for all $\chi$ as above. 

Let $\bar{\psi}$ be the reduction of $\psi$ modulo $\varpi$ and let $[\bar{\psi}]$ be the Teichm\"uller lift of $\bar{\psi}$, then the character 
$\psi^{-1} [\bar{\psi}]$ takes values in $1+\pp$ and thus we may take its square root by the usual binomial formula. If $\chi:= \sqrt{\psi^{-1} [\bar{\psi}]}$ then $\psi \chi^2= [\bar{\psi}]$ has order prime to 
$p$. Proposition \ref{equal_actions} imply that for all open pro-$p$ subgroups 
$V^\pp_p$ of $U^\pp_p$ the two actions on $S_{[\bar{\psi}]}(U^p, L/\OO)^{V^\pp_p}_{\mm}$ coincide. By passing to the direct limit we obtain that the two actions on $S_{[\bar{\psi}]}(U^p, L/\OO)$ coincide, by twisting we obtain the same result with $\psi$ instead of $[\bar{\psi}]$.
\end{proof}

\begin{cor} Let $x\in \mSpec \TT(U^p)_{\mm}[1/p]$ be such that the restriction of the corresponding 
Galois representation $\rho_x: G_{F,S}\rightarrow \GL_2(\kappa(x))$ to $G_{F_{\pp}}$ is irreducible. 
Then there is an isomorphism 
of $L$-Banach space representations of $G$: 
\begin{equation}\label{central}
(S_{\psi, \lambda}(U^{\pp}, \OO)\otimes_{\OO} L)[\mm_x]\cong \Pi^{\oplus n}
\end{equation}
for some integer $n\ge 0$, where $\Pi\in \Ban^{\adm}_{G}(L)$ corresponds to $\rho_x|_{G_{F_{\pp}}}$
via the $p$-adic local Langlands correspondence for $\GL_2(\Qp)$.
\end{cor}
\begin{proof} Let $y: R^{\ps}_{\tr \rbar}\rightarrow L$ be the homomorphism corresponding to 
$\tr( \rho_x|_{G_{F_{\pp}}})$. It follows from Corollaries 6.8, 8.14, 9.36 and Proposition 10.107 in \cite{image} that \eqref{central} holds for any $\Pi'\in \Ban^{\adm}_G(L)_{\rbarss}$ of finite length, 
on which $R^{\ps}_{\tr \rbar}$ acts as the center of the category via $y$. Corollary \ref{actions_coincide} and Lemma \ref{fibre_finite} below 
 imply that we may apply this observation to $\Pi'=(S_{\psi, \lambda}(U^{\pp}, \OO)\otimes_{\OO} L)[\mm_x]$.
\end{proof}

\begin{lem}\label{fibre_finite} We let $R^\ps_{\tr \rbar}$ act on $\tau\in \Mod^{\ladm}_G(\OO)_{\rbarss}$ as the centre
of the category $\Mod^{\ladm}_G(\OO)_{\rbarss}$. If the $G$-socle of $\tau$ is of finite length then $k \otimes_{R^\ps_{\tr \rbar}} \tau^{\vee}$ is of finite length in $\dualcat(\OO)$. In particular, the assertion holds for any admissible representation $\tau$ in $\Mod^{\ladm}_G(\OO)_{\rbarss}$.
\end{lem} 
\begin{proof} There are only finitely many isomorphism classes of irreducible objects in $\BB_{\rbarss}$. Let 
$\pi_1, \ldots, \pi_k$ be a set of representatives. For each $i$ let $\pi_i\hookrightarrow J_i$ be an injective 
envelope of $\pi_i$ in $\Mod^{\ladm}_G(\OO)$.  Dually we let $P_i:=J_i^{\vee}$ so that $P_i \twoheadrightarrow \pi_i^{\vee}$ is a projective
envelope of $\pi_i^{\vee}$ in $\dualcat(\OO)$. Since the category $\Mod^{\ladm}_G(\OO)_{\rbarss}$
 is locally finite we may embed $\tau$
into an injective envelope of its $G$-socle, which is isomorphic to a direct sum of finitely many copies of $J_i$. 
This implies that dually there is a surjection $\oplus_{i=1}^k P_i^{\oplus m_i}\twoheadrightarrow \tau^{\vee}$, 
where $m_i$ are finite multiplicities. Thus it is enough to prove the statement for $\tau=J_i$, $1\le i \le k$. 
It follows from \cite[Lemma 3.3]{image} that $k \otimes_{R^\ps_{\tr \rbar}} P_i$ is of finite length if and only if 
$\Hom_{\dualcat(\OO)}(P_j, k \otimes_{R^\ps_{\tr \rbar}} P_i)$ is a finite dimensional $k$-vector space for $1\le j\le k$. 
Since $P_j$ is projective the natural map
$$k \otimes_{R^\ps_{\tr \rbar}} \Hom_{\dualcat(\OO)}(P_j, P_i)\rightarrow \Hom_{\dualcat(\OO)}(P_j, k \otimes_{R^\ps_{\tr \rbar}} P_i)$$
is an isomorphism. Thus it is enough to show that $\Hom_{\dualcat(\OO)}(P_j, P_i)$ is a finitely generated $R^\ps_{\tr \rbar}$-module for $1\le i,j\le k$. This assertion follows from Proposition 6.3, Corollary 8.11 in conjuction with 
Proposition B.26, Corollaries 9.25 and 9.27, Lemma 10.90 in \cite{image}.
\end{proof}

\begin{lem}\label{finite_fibre_2} 
Let $\tau \in  \Mod^{\ladm}_G(\OO)_{\rbarss}$ and suppose that we are given a faithful action on $\tau$ of a local $R^{\ps}_{\tr \rbar}$-algebra $R$ with residue field $k$ via $R\hookrightarrow \End_G(\tau)$. Assume 
that the composition $R^{\ps}_{\tr \rbar}\rightarrow R \rightarrow \End_G(\tau)$ coincides 
with the action of $R^{\ps}_{\tr \rbar}$ on $\tau$ as the centre of category $\Mod^{\ladm}_G(\OO)_{\rbarss}$.
If the $G$-socle of $\tau$ is  of finite length then $R$ is finite over $R^{\ps}_{\tr \rbar}$.
\end{lem}
\begin{proof} Let $P_i$ be as in the proof of Lemma \ref{fibre_finite} and let $P=\oplus_{i=1}^n P_i$. Then $P$ is a 
projective generator for the category $\dualcat(\OO)_{\rbarss}$ and thus $R$ acts faithfully on 
$\md:=\Hom_{\dualcat(\OO)}(P, \tau^{\vee})$. The proof of Lemma \ref{fibre_finite} shows that 
$\md \otimes_{R^\ps_{\tr \rbar}} k$ is a finite dimensional $k$-vector space. Since $\md$ is a compact 
$R^{\ps}_{\tr \rbar}$-module we deduce that  $\md $ is finitely generated over $R^{\ps}_{\tr \rbar}$ and hence also 
over $R$. If $v_1, \ldots, v_n$ are generators of $\md$ as an $R$-module then the map $r\mapsto (r v_1, \ldots, r v_n)$ induces an embedding of $R$-modules $R\hookrightarrow \md^{\oplus n}$. Since $R^{\ps}_{\tr \rbar}$ is noetherian we deduce that $R$ is a finitely generated $R^{\ps}_{\tr \rbar}$-module. 
\end{proof} 

\begin{lem}\label{no_adm_q} Let $J$ be an injective envelope of an irreducible representation in $\Mod^{\ladm}_{G, \zeta}(k)$. Then $J$ 
does not admit an admissible quotient. 
\end{lem}
\begin{proof} The action of $R^{\ps}_{\tr \rbar}$ on $J$ as the centre of the category 
factors through the quotient $R^{\ps, \psi}_{\tr \rbar}/\varpi$.  If $\sigma$ is a smooth irreducible representation of 
$K$ then $\Hom_K(\sigma, J)^{\vee}$ is either zero or a finitely generated Cohen-Macaulay module of 
$R^{\ps, \psi}_{\tr \rbar}/\varpi$ of dimension $1$; in most cases this follows from
\cite[Thm.\,5.2]{duke}, since it implies that the element denoted by $x$ in \textit{loc.\,cit.} is a regular parameter. The rest of the cases are handled in \cite[Prop.\,3.9]{hu} by a similar argument. It is proved in \cite[Prop.\,5.16]{image} that 
$J$ is injective in $\Mod^{\sm}_{G, \zeta}(k)$. Since restriction is right adjoint to compact induction, which is
an exact functor, we deduce that $J$ is injective in $\Mod^{\sm}_{K, \zeta}(k)$. Thus if $\sigma \in \Mod^{\sm}_{K, \zeta}(k)$
is of finite length then $\Hom_K(\sigma, J)^{\vee}$ is either zero or a successive extension of $1$-dimensional Cohen--Macaulay 
modules corresponding to the irreducible subquotients $\sigma'$ of $\sigma$ for which $\Hom_K(\sigma', J)$ is non-zero. Hence, if $\Hom_K(\sigma, J)^{\vee}$  is non-zero then 
it is Cohen--Macaulay of dimension $1$. We note that  by taking $\sigma= \Ind_{K_n (Z\cap K)}^K \zeta$, where $K_n$ is the $n$-th congruence subgroup of $K$, we deduce that $(J^{K_n})^{\vee}$ is Cohen--Macaulay of dimension $1$. 

If  $\tau$ is a quotient of $J$ then we may choose $n\ge 1$ such that  
the map $J^{K_n}\rightarrow \tau^{K_n}$ is non-zero. By taking Pontryagin duals we obtain a non-zero map of  
$R^{\ps}_{\tr \rbar}$-modules $(\tau^{K_n})^{\vee} \rightarrow (J^{K_n})^{\vee}$. If $\tau$ is admissible then 
$\tau^{K_n}$ is a finite dimensional $k$-vector space, and hence $(J^{K_n})^{\vee}$ will contain a non-zero submodule 
killed by the maximal ideal of $R^{\ps}_{\tr \rbar}$. This is not possible as $(J^{K_n})^{\vee}$ is Cohen--Macaulay of dimension $1$.
\end{proof}

\begin{prop}\label{finite_fibre_3} Let $\tau$ and $R\hookrightarrow \End_G(\tau)$ be as in Lemma \ref{finite_fibre_2}.
Assume that $\tau$ is killed by $\varpi$ and has a central character. If $\tau$ is admissible then 
the Krull dimension of $R$ is less or equal to $2$.
\end{prop}
\begin{proof}  Since $\tau$ is admissible, its $G$-socle is of finite length. Lemma \ref{finite_fibre_2} implies that 
$R$ is a finite $R^{\ps}_{\tr \rbar}$-module. Since we are interested only in the Krull dimension we may assume that 
$R$ is the image of $R^{\ps}_{\tr \rbar}$ in $\End_G(\tau)$.  Since $\tau$ has a central character the map 
$R^{\ps}_{\tr \rbar}\rightarrow \End_G(\tau)$ factors through the quotient $R^{\ps}_{\tr \rbar}\twoheadrightarrow R^{\ps, \psi}_{\tr \rbar}$, which parameterises pseudocharacters with a fixed determinant.  Since $\varpi$ kills 
$\tau$ we get a surjection $R^{\ps, \psi}_{\tr \rbar}/\varpi\twoheadrightarrow  R$. Now $R^{\ps, \psi}_{\tr \rbar}/\varpi$ 
is an integral domain of Krull dimension $3$, see \cite[\S A]{image}, so we have to show that $R^{\ps, \psi}_{\tr \rbar}/\varpi$ does not act faithfully on $\tau$. %This assertion follows from \cite[Thm.\,3.23]{hu}. We will  sketch the proof of this result along the lines of the proof of \cite[Thm.\,3.23]{hu} for the convenience of the reader. We first translate the faithfulness of the action  into a commutative algebra statement. 

Let $\pi_1, \ldots \pi_k$ be as in the proof of Lemma \ref{fibre_finite}, but let $\pi_i\hookrightarrow J_i$ 
now denote an injective envelope of $\pi_i$ in $\Mod^{\ladm}_{G, \zeta}(k)$, so that in the rest of the proof we work with a fixed central character. Let $P:=\oplus_{i=1}^k P_i$, where $P_i:=J_i^{\vee}$ and
 let $E= \End_{\dualcat(k)}(P)$. Then $P$ is a projective generator of $\dualcat_\zeta(k)_{\rbarss}$ and 
 if we let $\md:=\Hom_{\dualcat(k)}(P, \tau^{\vee})$ then evaluation induces a natural isomorphism 
 $\md \wtimes_E P\overset{\cong}{\longrightarrow}\tau^{\vee}$. Thus it is enough to show that $R^{\ps, \psi}_{\tr \rbar}/\varpi$ does not act faithfully on $\md$. 
 
 The action of $R^{\ps, \psi}_{\tr \rbar}/\varpi$ on $\md$ is faithful if and only if $\md\otimes_{R^{\ps}_{\tr \rbar}} Q$ is non-zero, where $Q$ is the quotient field of $R^{\ps, \psi}_{\tr \rbar}/\varpi$.  Let us assume that this is the case. After replacing 
 $\md$ by a subquotient we may assume that $\md$ is a cyclic $E$-module and the map $\md \rightarrow \md\otimes_{R^{\ps}_{\tr \rbar}} Q$ is injective. We claim that the algebra $E\otimes_{R^{\ps}_{\tr \rbar}} Q$ is semisimple. Granting the claim we deduce that 
the surjection $E\otimes_{R^{\ps}_{\tr \rbar}} Q\twoheadrightarrow \md\otimes_{R^{\ps}_{\tr \rbar}} Q$ has 
a section of $E\otimes_{R^{\ps}_{\tr \rbar}} Q$-modules. By composing it with the injection $\md \hookrightarrow \md\otimes_{R^{\ps}_{\tr \rbar}} Q$ we obtain an injection $\md \hookrightarrow E\otimes_{R^{\ps}_{\tr \rbar}} Q$
of $E$-modules.  Since $\md$ is finitely generated over $E$ we may multiply this embedding by an element of 
 $R^{\ps, \psi}_{\tr \rbar}$ to obtain an injection $E$-modules $\md\hookrightarrow E$. 
 Since $P$ is a projective generator for the block,  by applying $\wtimes_E P$ we obtain an injection $\tau^{\vee}\hookrightarrow P$ and dually a surjection $J\twoheadrightarrow \tau$. Thus one of the $J_i$'s admits an admissible quotient contradicting Lemma \ref{no_adm_q}.

To prove the claim we observe that  if $\rbar$ does not have scalar semi-simplification then 
 $E\otimes_{R^{\ps}_{\tr \rbar}} Q$ is a matrix algebra over $Q$. If $\rbar$ is irreducible then this follows from 
 \cite[Prop.\,6.3]{image}, if $\rbar$ is reducible generic then this follows from \cite[Cor.\,8.11]{image} together with 
 \cite[Prop.\,4.3, 3.12]{2adic}, 
 if $\rbarss= \chi\oplus \chi \omega$ then this follows from the explicit description of the endomorphism ring of a projective 
 generator of $\dualcat_\zeta(\OO)_{\rbarss}$ given after \cite[Cor.\,10.94]{image}: both generators $c_0$ and $c_1$ of the reducible locus in $R^{\ps, \psi}_{\tr \rbar}$ are non-zero in $R^{\ps, \psi}_{\tr \rbar}/\varpi$ and hence in $Q$.

 Let us assume that
  $\rbar$ has scalar semi-simplification. After twisting we may assume that $\zeta$ is trivial and use \cite[Cor.\,9.27]{image} 
  to identify $E$ with the ring opposite to $R/\varpi$, where $R$ is the ring defined in \cite[Eqn.\,(125)]{image}. Then $R^{\ps, \psi}_{\tr \rhobar}$ gets identified with the subring denoted by $\OO[[t_1, t_2, t_3]]$ in \cite{image}. Let $C$ be the finite $R^{\ps, \psi}_{\tr \rhobar}$-algebra defined in \cite[Def.\,9.7]{image}, let $x$ be a generic point of $C\otimes_{R^{\ps}_{\tr \rhobar}} Q$ and 
  let $\kappa(x)$ be its residue field.  We claim that $E\otimes_{R^{\ps}_{\tr \rhobar}} \kappa(x)\cong M_2(\kappa(x))$. 
  The proof of Lemma 9.20 in \cite{image} shows that the element denoted by $(uv-vu)(uv-vu)^*$ is non-zero in $Q$. The 
  proof Lemma 9.21 in \cite{image} goes through to show that the specialisation at $x$ of the representation $\rho: E\rightarrow M_2(C)$, constructed in \cite[Prop.\,9.8]{image}, is absolutely irreducible over $\kappa(x)$. The double centralizer theorem 
  implies that the map $E\otimes_{R^{\ps}_{\tr \rhobar}} \kappa(x) \rightarrow M_2(\kappa(x))$ is surjective. Since
the algebra $E\otimes_{R^{\ps}_{\tr \rhobar}} \kappa(x)$ is $4$-dimensional as 
  a $\kappa(x)$-vector space, see \cite[Lem.\,9.18]{image}, we obtain the claim. 
  \end{proof}
  \begin{remar} One could give a different proof of the Proposition using \cite[Prop.\,5.6]{hu}.
  \end{remar}

\begin{cor}\label{ff_5} Let $\tau$ and $R\hookrightarrow \End_G(\tau)$ be as in Lemma \ref{finite_fibre_2}. Assume that $\tau$ has a central character 
and is $\varpi$-divisible (equivalently $\tau^{\vee}$ is $\varpi$-torsion free). Then the Krull dimension of $R$ is at most $3$. 
\end{cor}
\begin{proof}  Since $\tau^{\vee}$ is $\varpi$-torsion free we have an exact sequence 
$0\rightarrow \tau^{\vee} \overset{\varpi}{\rightarrow} \tau^{\vee}\rightarrow \tau^{\vee}/\varpi\rightarrow 0$. 
Let $P$ be the projective generator of $\dualcat_{\zeta}(\OO)_{\rbarss}$ as in the proof of Proposition \ref{finite_fibre_3}.
We have an exact sequence of $R$-modules 
$$0\rightarrow \Hom_{\dualcat(\OO)}(P, \tau^{\vee})\overset{\varpi}{\rightarrow}  \Hom_{\dualcat(\OO)}(P, \tau^{\vee})\rightarrow \Hom_{\dualcat(\OO)}(P, \tau^{\vee}/\varpi)\rightarrow 0.$$
As explained in the proof of Lemma \ref{finite_fibre_2}, $\Hom_{\dualcat(\OO)}(P, \tau^{\vee})$ is a finitely generated faithful $R$-module. If we let $\overline{R}$ be the quotient of 
$R$ which acts faithfully on $\Hom_{\dualcat(\OO)}(P, \tau^{\vee}/\varpi)$ then its Krull dimension is equal to the dimension of the support of $\Hom_{\dualcat(\OO)}(P, \tau^{\vee}/\varpi)$
in $\Spec R$, which is one less than the dimension of the support of $\Hom_{\dualcat(\OO)}(P, \tau^{\vee})$ in $\Spec R$, as $\varpi$ is regular on the module, 
and thus is equal to $\dim R -1$, as the action of $R$ on $\Hom_{\dualcat(\OO)}(P, \tau^{\vee})$ is faithful. 

It follows from Proposition above applied to $\overline{R}$ and $\tau[\varpi]$ that the Krull dimension of $\overline{R}$ is at most $2$, thus the Krull dimension of $R$ is at most $3$.
\end{proof} 

\begin{cor}\label{ff_6} The image of $\TT(U^p)_{\mm}$ in $\End_G^{\cont}( S_{\psi, \lambda}(U^\pp, \OO)_\mm)$
is a finite $R^{\ps}_{\tr \rbar}$-algebra of Krull dimension at most $3$.
\end{cor}
\begin{proof} This follows from Lemma \ref{finite_fibre_2}, Corollary \ref{ff_5} applied to $\tau=S_{\psi, \lambda}(U^{\pp}, L/\OO)_{\mm}$. 
\end{proof}

\begin{remar} The proof of  Corollary \ref{ff_6} goes back to the workshop on Galois representations and Automorphic forms in Princeton in 2011, and was motivated by discussions with Matthew Emerton there. It remained unpublished, since after we communicated the result to Frank Calegari, see \cite{persiflage}, together with Patrick Allen they  proved
 much more general results concerning finiteness of global deformation rings over local deformation rings, see \cite{AC}. However, one advantage of our argument is that we don't have to assume anything about the image of the global Galois representation $\rhobar$.
 \end{remar} 

\begin{lem}\label{dim_irred} If $\pi \in \Mod^{\ladm}_{G, \zeta}(\OO)$ is irreducible and not a character then $$\delta_{\OO[[K]]}(\pi^{\vee})=1.$$
\end{lem}
\begin{proof} This follows from the proof of Corollary 7.5 in \cite{schmidt_strauch}. Alternatively, one can use \cite[Prop.\,5.4, 5.7, Thm.\,5.13]{kohlhaase}. However, the proof dearest to the author's heart is to show that  $\dim \pi^{K_n}$ grows as $C p^n$, for some constant $C$,  as $K_n$ runs over principal 
congruence subgroups of $K$. If $\pi$ is principal series then this can be done by hand, the result for 
special series can be deduced from this. The most interesting case, when $\pi$ is supersingular, can be 
deduced from the exact sequence in Theorem 6.3 in \cite{ext_super}. Let us sketch the argument using the notation of \textit{loc.\,cit.}. Indeed, using the exact sequence it is enough to estimate the growth of $\dim M^{K_n}$.  \cite[Lem.\,4.10]{ext_super} implies that it is enough to estimate the growth of $\dim M_{\sigma}^{K_n}$ and $\dim M_{\tilde{\sigma}}^{K_n}$.
 By \cite[Prop. 4.7]{ext_super} the restrictions of $M_{\sigma}$ and $M_{\tilde{\sigma}}$ to $N_0:=\bigl(\begin{smallmatrix} 1 & \Zp \\ 0 & 1 \end{smallmatrix}\bigr )$ are isomorphic to the space of 
 smooth functions from $N_0$ to $k$. This implies that 
 $$\dim M_{\sigma}^{K_n \cap N_0}= \dim M_{\tilde{\sigma}}^{K_n \cap N_0}=p^n.$$
 This gives an upper bound on the growth of $K_n$-invariants, which has to be of the right order, since 
 both $M_{\sigma}$ and $M_{\tilde{\sigma}}$ are infinite dimensional $k$-vector spaces. 
 
 Yet another  alternative is to use results of Morra \cite{morra}, who actually computes the dimensions
 $\dim \pi^{I_n}$, where $I_n$ is a certain filtration of Iwahori subgroup by open normal subgroups.
\end{proof} 

\begin{prop}\label{find_flat}  Let $R$ be the image of $\T(U^p)_{\mm}\rightarrow \End_G^{\cont}( S_{\psi, \lambda}(U^\pp, \OO)_\mm)$. 
Then there is a subring $A\subset R$, such that $A\cong \OO[[x, y]]$, $R$ is finite over $A$ and $S_{\psi, \lambda}(U^\pp, \OO)_\mm^d$ is a flat $A$-module.
 \end{prop}
 \begin{proof} Lemmas \ref{admissible}, \ref{K-injective} and equation \eqref{more_duals} imply that $S_{\psi, \lambda}(U^\pp, \OO)_\mm^d$ is projective in $\Mod^{\pro}_{K, \zeta}(\OO)$ and 
is  finitely generated over $\OO[[K]]$. Thus $S_{\psi, \lambda}(U^\pp, \OO)_\mm^d/\varpi$ is isomorphic 
 as $\OO[[K_1]]$-module to a finite direct sum of copies of $k[[K_1/Z_1]]$, where $Z_1$ is the centre of $K_1$. This implies that $S_{\psi, \lambda}(U^\pp, \OO)_\mm^d$
 is a Cohen-Macaulay $\OO[[K_1]]$-module of dimension $4$. The fibre $\mathcal F:=k\otimes_R S_{\psi, \lambda}(U^\pp, \OO)_\mm^d$ is a quotient
 of $k\otimes_{R^{\ps}_{\tr \rbar}} S_{\psi, \lambda}(U^\pp, \OO)_\mm^d$. Corollary \ref{actions_coincide} and 
 Lemma \ref{fibre_finite} imply that $\mathcal F$ is of finite length as a $G$-representation.
 Lemma \ref{dim_irred} implies that the fibre has dimension less than or equal to one. If it is zero then all irreducible subquotients of $\mathcal F$ would be characters and by looking at the graded pieces of 
 the $\mm$-adic filtration on $S_\psi(U^\pp, \OO)_\mm^d$, where $\mm$ is the maximal ideal of $R$,
 we would deduce that all the irreducible subquotients of $S_\psi(U^\pp, \OO)_\mm^d$ are characters.
 Since the central character is fixed  and $p>2$ there are no non-trivial
  extensions between $1$-dimensional $G$-representations over $k$. This would imply that $\SL_2(\Qp)$ acts 
  trivially on $S_{\psi, \lambda}(U^\pp, \OO)_\mm^d$, which is impossible, since it would imply that $\SL_2(\Zp)$ acts trivially on a projective object in $\Mod^{\pro}_{K, \zeta}(\OO)$. Hence, the dimension of the fibre is $1$. 
  It follows from Corollary \ref{ff_6} that  the Krull 
 dimension of $R$ is at most $3$.  The assertion follows from Corollary \ref{find_A}.
   \end{proof} 
 
 \section{Main result}\label{section_main}
We keep the notation of the previous section. Let  $D$ be the quaternion algebra over $F$, which is 
ramified at $\pp$, split at $\infty_F$ and has the same ramification behaviour as $D_0$ at all the other places. 
We fix an isomorphism 
$$ D_0\otimes_F \mathbb A_F^{\pp, \infty_F}\cong D \otimes_F \mathbb A_F^{\pp, \infty_F}.$$
This allows us to view the subgroup $U^{\pp}$ of $(D_0\otimes_F \AfF)^\times$, considered in the previous section, as a subgroup of  $(D \otimes_F \AfF)^{\times}$. Let $D_{\pp}:=D\otimes_F F_{\pp}$. Then $D_{\pp}$ is the non-split quaternion algebra over $F_\pp=\Qp$. We let $U_{\pp}=\OO_{D_\pp}^{\times}$ and $U=U_{\pp}U^\pp$. If $K$ is an open subgroup of $U:=U_{\pp}U^\pp$ then we let $X(K)$ be the corresponding Shimura curve for $D/F$ defined over $F$.  We let 
$$\widehat{H}^i(U^{\pp}, L/\OO):=\varinjlim_{K_{\pp}} H^i_{\et}(X(K_{\pp} U^{\pp})_{\overline{F}}, L/\OO),$$
where the limit is taken over open subgroups of $U_{\pp}$. Theorem 6.2 in \cite{scholze} gives an isomorphism of $G_{\Qp}\times D_\pp^{\times}$-representations: 
\begin{equation}\label{scholze_0}
\mathcal{S}^i( S(U^{\pp}, L/\OO))\cong \widehat{H}^i(U^{\pp}, L/\OO),
\end{equation}
where $\mathcal{S}^i$ is the functor  $\pi \mapsto H^i_{\et}(\PP^1, \FF_\pi)$.  Let $v$ be a finite place of $F$ different from $\pp$ and let $g_v \in (D_0\otimes_F F_v)^\times$. 
We view $g_v$ also as an element of $(D\otimes_F F_v)^\times$ using the identification above. Multiplication with $g_v$ induces an isomorphism
$ S(U^{\pp}, L/\OO)\cong S( g_v^{-1} U^{\pp} g_v, L/\OO)$. It follows from the identifications explained at the end of the proof of \cite[Prop.\,6.5]{scholze} that the following diagram of 
$G_{\Qp}\times D_\pp^{\times}$-representations commutes: 
\begin{displaymath}
\xymatrix{\mathcal{S}^i( S(U^{\pp}, L/\OO))\ar[r]^-{\mathcal{S}^i( \cdot g_v)}\ar[d]^-{\cong}_-{\eqref{scholze_0}}&
\mathcal{S}^i( S(g_v^{-1}U^{\pp} g_v, L/\OO))\ar[d]^-{\cong}_-{\eqref{scholze_0}}\\
 \widehat{H}^i(U^{\pp}, L/\OO)\ar[r]^-{\cdot g_v} & \widehat{H}^i(g_vU^{\pp} g_v^{-1}, L/\OO).}
\end{displaymath}
Thus if we let 
$$\widehat{H}^i(U^p, L/\OO):=\varinjlim_{K_p} H^i_{\et}(X(K_p U^p)_{\overline{F}}, L/\OO),$$
where the limit is taken over open subgroups of $U_p$, then we deduce that \eqref{scholze_0} induces an isomorphism of $\TT^{\univ}_S[G_{\Qp}\times (D\otimes_{\QQ} \Qp)^{\times}]$-modules:
\begin{equation}\label{scholze_1}
\mathcal{S}^i( S(U^p, L/\OO))\cong \widehat{H}^i(U^p, L/\OO).
\end{equation}

Let $\rhobar: G_{F, S}\rightarrow \GL_2(k)$ be an absolutely irreducible representation as in the previous section and let $\mm$ be the corresponding maximal ideal in $\TT^{\univ}_S$. Let $$\rbar:= \rhobar|_{G_{F_{\pp}}}$$ and recall that we assume that $F_\pp=\Qp$. It follows from 
Corollary 7.5 in \cite{scholze} that \eqref{scholze_1} induces an isomorphism of $\TT(U^p)_{\mm}[G_{\Qp}\times (D\otimes_\QQ \Qp)^{\times}]$-modules:
\begin{equation}\label{scholze_3}
\mathcal{S}^1( S(U^p, L/\OO)_{\mm})\cong \widehat{H}^1(U^p, L/\OO)_{\mm}.
\end{equation}

\begin{lem}\label{revision} 
Let $\lambda$ be a continuous representation of $U_p^{\pp}$ on a finite free $\OO$-module and let 
$I$ be an ideal of $\TT(U^p)_{\mm}$ then \eqref{scholze_3} induces an injection
\begin{equation}\label{revision1}
 \mathcal S^1( \Hom_{U_p^{\pp}}(\lambda, S(U^p, L/\OO)_{\mm} [I]) \subset \Hom_{U_p^{\pp}}(\lambda, 
\widehat{H}^1(U^p, L/\OO)_{\mm})[I],
\end{equation}
such that the subgroup of $\OO_{D_{\pp}}^{\times}$ of elements of reduced norm $1$ acts trivially on the cokernel. Moreover, 
if the semi-simpification of $\rbar$ is not of the form $\chi \oplus \chi \omega$ then the cokernel is zero.
\end{lem}
\begin{proof} If $\lambda$ is the trivial representation then the first assertion is \cite[Prop.\,7.7]{scholze}. 
The assertion for general $\lambda$ follows from this by presenting $\lambda$ as $\OO[[U_p^{\pp}]]$-module and 
arguing as in \eqref{beijing_3} and \eqref{beijing_4}. 

If \eqref{revision1} is not an isomorphism then 
it follows from the proof of \cite[Prop.\,7.7]{scholze} that (after replacing $U^{\pp}$ with an open subgroup) there is a subquotient of $S(U^{\pp}, L/\OO)_{\mm}$ 
with non-zero $\SL_2(\Qp)$-invariants. Since this representation is locally admissible, we deduce that 
there $S_{\psi}(U^{\pp}, L/\OO)_{\mm}$ has a subquotient  with non-zero $\SL_2(\Qp)$-invariants, for some
character $\psi: (\AfF )^{\times}/F^{\times}\rightarrow \OO^{\times}$. Proposition \ref{right_object} implies 
that the block corresponding to $\rbarss$ contains a character. Thus we are in the setting of \cite[\S10]{image}
and so $\rbarss\cong \chi \oplus \chi \omega$.
\end{proof}

Let $\psi: (\AfF )^{\times}/F^{\times}\rightarrow \OO^{\times}$ be a continuous character  such that $\psi$ is trivial of  $(\AfF )^{\times}\cap U^p$. To ease the notation we will use the same symbol to denote the 
restriction of $\psi$ to the intersection of $(\AfF)^{\times}$ with various subgroups of $(D\otimes_F \AfF)^{\times}$ with the exception of $\zeta:= \psi|_{F_{\pp}^{\times}}$. We will also view $\psi$ as 
a character of $G_{F,S}$ via the class field theory and denote by the same letter its restriction to various
decomposition groups.

 Let $\widehat{H}^1_{\psi}(U^p, L/\OO)$ be the maximal submodule of 
$\widehat{H}^1(U^p, L/\OO)$ on which $(\AfF)^{\times}$ acts
by the character $\psi$. By Chebotarev's density theorem this coincides with the common eigenspace of all Hecke operators $S_v$ for the eigenvalue $\psi(\Frob_v)$.
% Hence, \eqref{scholze_1} induces an isomorphism of $\TT^{\univ}_S[G_{\Qp}\times (D\otimes_\QQ \Qp)^{\times}]$-modules:
%\begin{equation}\label{scholze_2}
%\mathcal{S}^1( S_{\psi}(U^p, L/\OO))\cong \widehat{H}^1_{\psi}(U^p, L/\OO).
%\end{equation}

If $\lambda$ is a continuous representation of $U_p^{\pp}$ on a finite free $\OO$-module with central character $\psi$ then we let 
$$\widehat{H}^1_{\psi, \lambda}(U^{\pp}, L/\OO)_{\mm}:=\Hom_{U^{\pp}_p}(\lambda, \widehat{H}^1_{\psi}(U^p, L/\OO)_{\mm}).$$
%Since $\rhobar$ is irreducible by assumption, $S_{\psi}(U^p, L/\OO)_{\mm}$ does not have any $\SL_2(\Qp)$-invariants, and it follows from \eqref{scholze_3}  that by applying $\mathcal S^1$ to \eqref{beijing_3} we obtain an isomorphism 

\begin{lem} The map  \eqref{scholze_3} induces  isomorphisms of
$\TT(U^p)_{\mm}[G_{\Qp}\times D^{\times}_{\pp}]$-modules:
\begin{equation}\label{scholze_4}
\mathcal{S}^1( S_{\psi, \lambda}(U^\pp, L/\OO)_{\mm}) \cong  \widehat{H}^1_{\psi, \lambda}(U^\pp, L/\OO)_{\mm}.
\end{equation}
and of 
$\TT(U^p)_{\mm}[G_{\Qp}\times (D\otimes_\QQ \Qp)^{\times}]$-modules:
\begin{equation}\label{scholze_2}
\mathcal{S}^1( S_{\psi}(U^p, L/\OO))\cong \widehat{H}^1_{\psi}(U^p, L/\OO).
\end{equation}

\end{lem}
\begin{proof}
It follows from \cite[Prop.\,5.6]{newton_pJL} that $\widehat{H}^1_{\psi}(U^p, L/\OO)_{\mm}$ is injective in the category $\Mod^{\sm}_{U_p, \psi}(\OO)$. Arguing as in the proof of Lemma \ref{injective} we obtain that $\widehat{H}^1_{\psi, \lambda}(U^\pp, L/\OO)_{\mm}$ is injective
in $\Mod^{\sm}_{U_{\pp}, \zeta}(\OO)$. 
Hence, if $H$ is an open pro-$p$ subgroup of 
$$D_{\pp}^{\times,1}:=\{ g \in D_{\pp}^{\times}: \Nrd(g)=1\},$$ 
which intersects the centre of $D_{\pp}^{\times}$ trivially, $(\widehat{H}^1_{\psi, \lambda}(U^\pp, L/\OO)_{\mm})^{\vee}$
is a free $\OO[[H]]$-module of finite rank. Since $D_{\pp}^{\times}$ is a $p$-adic analytic group, we may choose 
$H$ to be torsion free, in which case $\OO[[H]]$ is an integral domain, and thus does not contain non-zero 
$\OO$-submodule on which $H$ acts trivially.  Thus, the cokernel of \eqref{scholze_3}, applied with $I$ equal  to the 
ideal generated by $S_v- \psi(\Frob_v)$ for all $v\not \in S$,  is zero and thus we get \eqref{scholze_4}. The 
last isomorphism follows from \eqref{scholze_4} using \eqref{beijing_5}.
\end{proof}

 Arguing as in \cite[Lemma 5.3.8]{emerton_lg} we obtain isomorphisms
\begin{equation}\label{emerton_1}
(\widehat{H}^1_{\psi}(U^p, L/\OO)_{\mm})^{K_p}[\varpi^n]\cong H^1_{\psi}(X(U^pK_p), \OO/\varpi^n)_{\mm},
\end{equation}
\begin{equation}\label{emerton_2}
(\widehat{H}^1_{\psi}(U^p, L/\OO)_{\mm})^{K_p^{\pp}}[\varpi^n]\cong \widehat{H}^1_{\psi}(U^pK^{\pp}_p, \OO/\varpi^n)_{\mm}.
\end{equation}
Thus if we let 
$$\widehat{H}^1_{\psi}(U^p, \OO)_{\mm}:=\varprojlim_n \widehat{H}^1_{\psi}(U^p, \OO/\varpi^n)_{\mm}, \quad \widehat{H}^1_{\psi, \lambda}(U^\pp, \OO)_{\mm}:= \varprojlim_n 
\widehat{H}^1_{\psi, \lambda}(U^p, \OO/\varpi^n)_{\mm},$$
then these are $\OO$-torsion free and by inverting $p$ we obtain admissible unitary Banach space representations of $(D\otimes_{\QQ}\Qp)^{\times}$ and $D_{\pp}^{\times}$ respectively. Moreover, 
we have natural homeomorphisms of $\OO$-modules
\begin{equation}\label{even_more_duals}
(\widehat{H}^1_{\psi}(U^p, \OO)_{\mm})^d\cong (\widehat{H}^1_{\psi}(U^p, L/\OO)_{\mm})^{\vee}, \quad (\widehat{H}^1_{\psi, \lambda}(U^\pp, \OO)_{\mm})^d\cong (\widehat{H}^1_{\psi, \lambda}(U^\pp, L/\OO)_{\mm})^{\vee}.
\end{equation}
In the following to ease the notation we will omit the outer brackets in the duals. 
\begin{prop}\label{scholze_compatible} There is an isomorphism of $\TT(U^p)_\mm[G_{\Qp}\times D_{\pp}^{\times}]$-modules 
$$\Shat^1( S_{\psi, \lambda}(U^{\pp}, \OO)_{\mm}^d)\cong \widehat{H}^1_{\psi, \lambda}( U^\pp, \OO)_{\mm}^d.$$
\end{prop}
\begin{proof} This follows from the definition of the functor $\Shat^1$ together with \eqref{scholze_4} and \eqref{even_more_duals}.
\end{proof}

\begin{prop}\label{projective} $\widehat{H}^1_{\psi}( U^p, \OO)_{\mm}^d$ is a finitely generated $\OO[[U_p]]$-module, which is projective in $\Mod^{\pro}_{U_p, \psi}(\OO)$. 
$\widehat{H}^1_{\psi, \lambda}( U^\pp, \OO)_{\mm}^d$ is a finitely generated $\OO[[U_{\pp}]]$-module, which is projective in  $\Mod^{\pro}_{U_\pp, \zeta}(\OO)$.
\end{prop}
\begin{proof} The projectivity and finite generation follow  from \eqref{even_more_duals} together with injectivity and admissibility statements for $\widehat{H}^1_{\psi }(U^p, L/\OO)_{\mm}$ and $\widehat{H}^1_{\psi, \lambda}(U^\pp, L/\OO)_{\mm}$ explained above.
\end{proof}

 We assume $\psi$ and $\rhobar$ are such that $S_{\psi}(U^p, L/\OO)_{\mm}\neq 0$, see Lemma \ref{K-injective}. After  twisting by a character we may assume that the restriction of $\psi$ to $(\AfF)^{\times}\cap U_p$ is  a locally algebraic character. 

 \begin{thm}\label{non-vanish} 
The functor $\Sc^1$ is not identically zero on $\Mod^{\ladm}_G(\OO)_{\rbarss}$.
\end{thm}
\begin{proof} Let $\psi$ be such that $S_{\psi}(U^{p}, L/\OO)_\mm$ is non-zero and the restriction 
of $\psi$ to $(\AfF)^{\times}\cap U_p$ is  a locally algebraic character. We may write $\psi= \psi_{\sm} \psi_{\alg}$, 
where $\psi_{\sm}$ is a smooth character of $(\AfF)^{\times}\cap U_p$  and $\psi_{\alg}$ is the restriction 
to $(\AfF)^{\times}\cap U_p$ of the central character of an irreducible algebraic representation $W_{\alg}$ evaluated at $L$ 
of the algebraic group $(\Res^{F}_{\QQ} D_0^{\times})\otimes_{\QQ} L$, where $\Res$ denotes the restriction of scalars. 

Since $S_{\psi}(U^p, L/\OO)_\mm$  is an object 
of $\Mod^{\ladm}_G(\OO)_{\rbarss}$ by Proposition \ref{right_object}, using \eqref{scholze_2} 
it is enough to check that $\widehat{H}^1_{\psi}(U^p, L/\OO)_{\mm}$ is non-zero, and \eqref{even_more_duals} implies that it is enough to check that $\widehat{H}^1_\psi(U^p, \OO)_{\mm}$
is non-zero. The locally algebraic vectors in the Banach spaces $\widehat{H}^1_\psi(U^p, \OO)_{\mm}\otimes_{\OO} L$ and  $S_{\psi}(U^p, \OO)_{\mm}\otimes_{\OO} L$ are related to classical automorphic forms on $D^{\times}$ and $D_0^{\times}$, respectively, see Theorem 5.3 in
\cite{newton_pJL} and \cite[\S 3]{interpolate}. The assumption that $S_{\psi}(U^p, L/\OO)_{\mm}\neq 0$ implies that $S_{\psi}(U^p, \OO)_{\mm}\otimes_{\OO} L$ is non-zero, and 
if $V_p$ is an open pro-$p$ subgroup of of $U_p$ then it follows from Lemma \ref{K-injective} that $S_{\psi}(U^p, \OO)_{\mm}\otimes_{\OO} L$ as 
$V_p$-representation is isomorphic to a finite direct sum of copies of $C_{\psi}(V_p, L)$ in the notation of Section \ref{lg}. Thus if $\gamma$ is a representation of 
$V_p$ on a finite dimensional $L$-vector space with the central character $\psi$ then $\Hom_{V_p}(\gamma, S_{\psi}(U^p, \OO)_{\mm}\otimes_{\OO} L)$ is non-zero.

We choose $V_p$ of the form $V_p^{\pp}\times V_{\pp}$, such that $\psi_{\sm}$ is trivial on $V_p \cap (\AfF)^{\times}$ and 
$V_{\pp}= \bigl ( \begin{smallmatrix} 1+ \pp^m & \pp^{m-1} \\ \pp^m & 1+\pp^m\end{smallmatrix} \bigr)$, for some $m\ge 1$. 
Let $\theta: V_{\pp} \rightarrow L^{\times}$ be the character, which maps   $\bigl ( \begin{smallmatrix} 1+ p^m a& p^{m-1}b  \\ p^m c & 1+p^m d\end{smallmatrix} \bigr)$ to 
$\alpha( b+c)$, where $\alpha: \OO_{F_{\pp}}/ \pp^{m-1} \rightarrow L^{\times}$ is any non-trivial additive character. Then $\theta$ is trivial on $V_{\pp} \cap (\AfF)^{\times}$ and is 
a supercuspidal type, by which we mean that if $\pi$ is a smooth irreducible representation of $\GL_2(F_{\pp})$ and $\Hom_{V_{\pp}}( \theta, \pi)\neq 0$ then $\pi$ is supercuspidal, see 
\cite[Prop.\,7.1]{scholze} or \cite[Prop.\,3.19]{jep} for a more general setting. We extend $\theta$ to a character of $V_p$ by mapping $V_p^{\pp}$ to $1$. If we let 
$\gamma:= \theta \otimes W_{\alg}$ as a representation of $V_p$ then it has a central character $\psi$ by construction and thus $\Hom_{V_p}(\gamma, S_{\psi}(U^p, \OO)_{\mm}\otimes_{\OO} L)$ 
is non-zero as explained above.  An eigenvector for the Hecke operators on this finite dimensional vector space will give us a classical  automorphic form on $D_0^{\times}$. Since $\theta$ is a supercuspidal 
type the automorphic form will be  supercuspidal at $\pp$, see the proof of \cite[Thm.\,5.1]{jep} or the proof of \cite[Cor.\,7.3]{scholze}. Since the automorphic form is supercuspidal at $\pp$, we may transfer it to an automorphic 
form on $D^{\times}$ by the classical Jacquet--Langlands correspondence, which in turn gives us a
non-zero vector in the locally algebraic vectors of $\widehat{H}^1_\psi(U^p, \OO)_{\mm}\otimes_{\OO} L$.
\end{proof}

\begin{remar} If $\rbar$ is irreducible then $\BB_{\rbar}$ consists of one isomorphism class 
of a supersingular representation $\pi$. It follows from Theorem \ref{non-vanish} that 
$\mathcal{S}^1(\pi)\neq 0$. This implies that $\mathcal{S}^1(\pi')\neq 0$ for any $\pi'\in \Mod^{\ladm}_{G}(\OO)_{\rbar}$.
\end{remar}

From now on we assume that $\rbarss=\chi_1\oplus \chi_2$ with $\chi_1\chi_2^{-1}\neq \omega^{\pm 1}$. Let $\pi_1$, $\pi_2$ be the principal series representations defined in \eqref{def_pi}. Proposition \ref{enough} implies that at least one of $\Shat^1(\pi_1^{\vee})$ 
and $\Shat^1(\pi_2^{\vee})$ is non-zero.

 For a finitely generated $\OO[[U_{\pp}]]$-module $M$  
we will abbreviate $\delta(M):=\delta_{\OO[[U_{\pp}]]}(M)$. If $\mathrm B$ is an admissible  unitary 
Banach space representation of $D_{\pp}^{\times}$, then we choose an open bounded $D_{\pp}^{\times}$-invariant lattice $\Theta$ in $\mathrm B$ and an open uniform pro-$p$ group $K$ of $D_{\pp}^{\times}$. Then $\Theta^d\otimes_{\OO} L$ is a finitely generated module over the Auslander regular ring $\OO[[K]]\otimes_\OO L$. We let $\delta(\mathrm B)$ be the dimension of 
$\Theta^d\otimes_{\OO} L$ over $\OO[[K]]\otimes_\OO L$ and note that it is equal to $\delta_{\OO[[K]]}( \Theta^d)-1$.  We will sometimes refer to $\delta(M)$ and $\delta(\mathrm B)$ as the $\delta$-\textit{dimension}.

\begin{prop}\label{maximum} The maximum of $\delta(\Shat^1(\pi_1^{\vee}))$ and $\delta(\Shat^1(\pi_2^{\vee}))$ is equal to $1$.
\end{prop}
\begin{proof} Let $A$ be the ring in Proposition \ref{find_flat} and let $K$ be an open uniform pro-$p$ subgroup of $D_\pp^\times$. It follows from Propositions \ref{test_flatness} and \ref{scholze_compatible} that 
$\widehat{H}^1_{\psi, \lambda}( U^\pp, \OO)_{\mm}^d$ is $A$-flat.  Since
it is non-zero, the fibre  $\mathcal F:=k \otimes_A \widehat{H}^1_{\psi, \lambda}( U^\pp, \OO)_{\mm}^d$ is also non-zero and \eqref{eqn_dim} implies that its $\delta$-dimension is equal to 
$\delta(\widehat{H}^1_{\psi, \lambda}( U^\pp, \OO)_{\mm}^d)- 3.$
It follows from Proposition \ref{projective} that the $\delta$-dimension of 
$\widehat{H}^1_{\psi, \lambda}( U^\pp, \OO)_{\mm}^d$ is equal to $4$, see the argument in the proof of the analogous statement for $S_{\psi, \lambda}(U^{\pp}, \OO)_{\mm}^d$ in Proposition \ref{find_flat}.  Hence the $\delta$-dimension of the fibre is equal to $1$.
 
 As explained in the proof of that Proposition \ref{find_flat} the fibre $k\otimes_A S_{\psi, \lambda}(U^{\pp}, \OO)_{\mm}^d$ is of finite length in $\dualcat(\OO)_{\rbarss}$ and all irreducible 
 subquotients are isomorphic to either $\pi^{\vee}_1$ or $\pi_2^{\vee}$. Moreover, both $\pi^{\vee}_1$ and $\pi_2^{\vee}$ 
 occur as subquotients. Since $\Shat^1$ is exact by Corollary \ref{ludwig_cor} we deduce that $\mathcal F$ has a filtration of finite length with graded pieces 
 isomorphic to either $\Shat^1(\pi_1^{\vee})$ or $\Shat^1(\pi_2^{\vee})$, which implies the assertion. 
\end{proof}

\begin{cor}\label{delta_Pi} Let $r: G_{\Qp}\rightarrow \GL_2(L)$ be a continuous representation 
with $\det r= \psi \varepsilon^{-1}$ and $\rbarss= \chi_1\oplus \chi_2$. Let $\Pi\in \Ban^{\adm}_{G}(L)$
correspond to $r$ via the $p$-adic local Langlands correspondence for $\GL_2(\Qp)$. Then 
$\Shat^1(\Pi)\neq 0$ and $\delta(\Shat^1(\Pi))=1$. 
\end{cor}
\begin{proof} Let $\Theta$ be an open bounded $G$-invariant lattice in $\Pi$. Then 
$(\Theta\otimes_{\OO} k)^{\mathrm{ss}} \cong \pi_1\oplus \pi_2$, see \cite[\S 11]{image}. We may assume that
$\Theta\otimes_{\OO} k$ is an extension of $\pi_1$ by $\pi_2$. Then $\Shat^1(\Theta^d/\varpi)$ is an extension of  $\Shat^1(\pi_2^{\vee})$ by $\Shat^1(\pi_1^{\vee})$. Proposition \ref{maximum} implies that 
$\delta( \Shat^1(\Theta^d/\varpi))=1$. Proposition \ref{test_flatness} applied with $A=\OO$ implies that
$\Shat^1(\Theta^d)$ is $\OO$-torsion free and 
$\delta(\Shat^1(\Theta^d))= \delta(\Shat^1(\Theta^d/\varpi))+1=2$.  Hence, $\delta(\Shat^1(\Pi))=1$.
\end{proof}

\begin{lem}\label{dim0implies}Let $\mathrm B$ be an admissible unitary $D_\pp^\times$-representation.   Then $\mathrm B$ is a finite dimensional $L$-vector space if and only if its $\delta$-dimension is $0$.  \end{lem}
\begin{proof} Let $K$ be an open uniform pro-$p$ subgroup of $D_\pp^\times$ and let $\Theta$  be an open bounded $D_\pp^\times$-invariant lattice in $\mathrm B$. Then
$$\delta(\mathrm B)= \delta(\Theta^d)-1 = \delta( \Theta^d/ \varpi)$$
and the assertion follows from Lemma \ref{dim_0}.
 \end{proof}

\begin{cor} If  $\Pi$ is  as in Corollary \ref{delta_Pi}  then $\Shat^1(\Pi)$ is of finite length 
  in the category of admissible unitary $L$-Banach space representations of $D_{\pp}^{\times}$ if and only if it has finitely many irreducible subquotients, which are finite dimensional as $L$-vector 
  spaces. 
 \end{cor}
 \begin{proof} Since $\Shat^1(\Pi)$ is admissible it has an irreducible subrepresentation, which we denote by $\mathrm B_1$, see \cite[Lem.\,5.8]{comp}. If $\Shat^1(\Pi)$ is not of finite length then by repeating the argument we obtain an ascending chain of Banach space subrepresentations $\{\mathrm B_i\}_{i\ge 0}$ such that $\mathrm B_0=0$ and the quotients $\mathrm B_{i+1}/\mathrm B_i$ for $i\ge 0$ are irreducible.  Let $K$ be a compact open subgroup of $D_{\pp}^{\times}$ without torsion. Then $M_i:=(\Shat^1(\Pi)/ \mathrm B_i)^d$ for $i\ge 0$ is a descending chain of finitely generated modules over a noetherian Auslander regular ring $\OO[[K]]\otimes_{\OO} L$.  According to \cite[Thm.\,4.2]{le} there is $n_0$ such that 
 $$ \delta(M_i/M_{i+1})\le \delta(M_0)-1= 0, \quad \forall i\ge n_0,$$
 where the last equality follows from Corollary \ref{delta_Pi}. Lemma \ref{dim0implies} implies that the quotients $\mathrm B_{i+1}/\mathrm B_i$ for $i\ge n_0$ are finite dimensional $L$-vector spaces. This contradicts the assumption that there
 are only finitely many such subquotients. 
 \end{proof}

\begin{thm}\label{main} Let $x\in \mSpec \TT(U^p)_{\mm}[1/p]$ be such that the restriction of the corresponding 
Galois representation $\rho_x: G_{F,S}\rightarrow \GL_2(\kappa(x))$ to $G_{F_{\pp}}$ is irreducible. 
Then $(\widehat{H}^1_{\psi, \lambda}(U^{\pp}, \OO)_{\mm}\otimes_\OO L)[\mm_x]$  is non-zero if and only if 
$(S_{\psi, \lambda}(U^\pp, \OO)_{\mm}\otimes_\OO L )[\mm_x]$ is non-zero.

In this case, there is an isomorphism of admissible unitary $\kappa(x)$-Banach space representations
of $G_{F_\pp}\times D_{\pp}^{\times}$: 
\begin{equation}\label{iso}
(\widehat{H}^1_{\psi, \lambda}(U^{\pp}, \OO)_{\mm}\otimes_{\OO} L)[\mm_x]\cong \Shat^1(\Pi)^{\oplus n},
\end{equation}
where $\Pi$ is the absolutely irreducible 
$\kappa(x)$-Banach space representation corresponding to $\rho_x|_{G_{F_\pp}}$ via the $p$-adic local Langlands correspondence for $\GL_2(\Qp)$.
In particular, the $\delta$-dimension of $(\widehat{H}^1_{\psi, \lambda}(U^{\pp}, \OO)_{\mm}\otimes_\OO L)[\mm_x]$ is $1$.
\end{thm}
\begin{proof} The assumption on $\rbar$ implies that \eqref{revision1} is an isomorphism. It follows from Proposition \ref{scholze_compatible} that 
$$\Shat^1((S_{\psi, \lambda}(U^\pp, \OO)_{\mm}\otimes_\OO L )[\mm_x])\cong (\widehat{H}^1_{\psi, \lambda}(U^{\pp}, \OO)_{\mm}\otimes_{\OO} L)[\mm_x],$$
as $G_{F_{\pp}}\times D_{\pp}^\times$-representations. The isomorphism \eqref{iso} is obtained by 
applying $\Shat^1$ to \eqref{central}. The assertion about the dimension follows from Corollary \ref{delta_Pi}.
\end{proof} 

\begin{remar} Since $\mm_x$ contains the ideal generated by $ S_v- \psi(\Frob_v)$ for all $v\not \in S$
Theorem \ref{main} implies Theorem \ref{thm_C} stated in the introduction.
\end{remar}

\begin{prop}\label{1-lalg} A unitary $D_\pp^\times$-representation $\mathrm B$ on a finite dimensional 
$L$-vector space with a central character is semisimple. Moreover,
$$\mathrm B \cong \bigoplus_{i=1}^n\bigoplus_{j=1}^m \Sym^{a_i} L^2 \otimes \dt^{b_i} \otimes \tau_j\otimes \eta_{ij}\circ \Nrd,$$
where $a_i\in \ZZ_{\ge 0}$, $b_i\in \ZZ$, $\tau_j$ is irreducible smooth, and $\eta_{ij}: \Qp^{\times}\rightarrow \OO^{\times}$ is a character. If the central character $\zeta_{\mathrm B}$ is locally algebraic then we may take $\eta_{ij}$ to be either trivial or $\eta_{ij}(x)= \sqrt{\pr(x|x|)}$, where $\pr: \Zp^{\times}\rightarrow 1+p \Zp$ denotes the projection.
\end{prop}
\begin{proof} We closely follow the proof of Proposition 3.2 of \cite{finite}. We first note that  any continuous action of a $p$-adic analytic group on a finite dimensional $L$-vector space
is automatically locally analytic, and hence induces an action of the universal enveloping algebra  
of its Lie algebra. Let $D_{\pp}^{\times, 1}=\{ g \in D_{\pp}^{\times}: \Nrd(g)=1\}$ and let $\mathfrak h$ be its Lie algebra. 
The Lie algebra of $D_\pp^\times$ is just $D_{\pp}$. We assume $L$ to be sufficiently large so that $D_{\pp}$ splits over $L$. Thus we have an isomorphism of $L$-Lie algebras 
$\mathfrak h\otimes_{\Qp} L \cong \slt$. This induces an isomorphism of enveloping algebras  
$U(\mathfrak h)\otimes_{\Qp} L\cong U(\slt)$ and an embedding  $D_{\pp}^{\times, 1}\hookrightarrow \SL_2(L)$. Let $J$ be the kernel of the 
map $U(\slt)\rightarrow \End_L(\mathrm B)$.  Since $\mathrm B$ is a finite dimensional $L$-vector space the codimension of  $J$ is finite and as explained in the proof of Proposition 3.2  of \cite{finite}
the algebra $U(\slt)/J$ is semi-simple and evaluation induces an isomorphism of $U(\slt)$-modules:
\begin{equation}\label{PGL3Q2}
\bigoplus_{a\ge 0} \Sym^a L^2\otimes_L \Hom_{U(\slt)}( \Sym^a L^2, \mathrm B)\overset{\cong}{\longrightarrow} \mathrm B.
\end{equation}
We may upgrade this isomorphism to an isomorphism of $D_{\pp}^{\times, 1}$-representations as follows. We let  $D_{\pp}^{\times, 1}$ act on 
$\Sym^a L^2$ via the embedding  $D_{\pp}^{\times, 1}\hookrightarrow \SL_2(L)$ above and on $\Hom_L(\Sym^a L^2, \mathrm B)$ by conjugation. 
Since $\Hom_L(\Sym^a L^2, \mathrm B)$ is a finite dimensional $L$-vector space, the corresponding representation is locally analytic. 
It follows from Proposition 2.1 in \cite{finite} that the smooth vectors for this action are equal to $\Hom_{U(\slt)}( \Sym^a L^2, \mathrm B)$, which makes it into 
a smooth representation of $D_{\pp}^{\times, 1}$. If we put the diagonal action of $D_{\pp}^{\times, 1}$ on $\Sym^a L^2\otimes_L \Hom_{U(\slt)}( \Sym^a L^2, \mathrm B)$ then 
the evaluation map is $D_{\pp}^{\times, 1}$-equivariant and hence the same holds for \eqref{PGL3Q2}. Since $D_{\pp}^{\times, 1}$ is a compact group, the category of 
smooth representations on $L$-vector spaces is semisimple. Moreover, if $\tau$ is an irreducible smooth 
representation of $D_{\pp}^{\times, 1}$, then  $\Sym^a L^2\otimes \tau$ is irreducible by Proposition 3.4 in \cite{finite}. 
Thus we have shown that the restriction of $\mathrm B$ to $D_{\pp}^{\times, 1}$ is semisimple. Since the centre $F^{\times}_\pp$ acts by a central character by assumption, the restriction of $\mathrm B$ to $D_{\pp}^{\times, 1} F^{\times}_\pp$ 
is semisimple. Since $D_{\pp}^{\times, 1} F^{\times}_\pp$ is of finite index in $D_\pp^\times$, this implies that $\mathrm B$ is a 
semisimple representation of $D_\pp^\times$.

Let us assume that $\mathrm B$ is absolutely irreducible, and let $a$ be such that 
$$\Hom_{U(\slt)}(\Sym^a L^2, \mathrm B)\neq 0.$$
Let $\pr: \OO^{\times}\rightarrow 1+\pp$ denote the projection to the principal units. We use the same symbol for $\Zp^{\times}\rightarrow 1+p\Zp$.
Since $p> 2$ 
we may define a continuous square root on $1+\pp$ by the usual binomial formula. Let $\eta: \Qp^{\times}\rightarrow \OO^{\times}$ be the character 
$\eta(x)= \sqrt{ \pr(\zeta_{\mathrm B}(x))^{-1} \pr(x|x|)^a}$, where $\zeta_{
\mathrm B}$ is the central character of $\mathrm B$.
Then the restriction of the central character of $\mathrm B \otimes \eta\circ\Nrd$ to $1+p \Zp$ is 
equal to $x\mapsto x^a$. If $H$ is an open subgroup of $D_{\pp}^{\times, 1}$ then $(1+p \Zp)H$ is an open subgroup of $D_\pp^{\times}$ and so their Lie 
algebras will coincide. Hence, $\Hom_{U(\glt)}(\Sym^a L^2, \mathrm B \otimes \eta\circ\Nrd)\neq 0$,
and arguing as before we conclude that $\mathrm B \otimes \eta\circ\Nrd\cong \Sym^a L^2 \otimes \tau$, 
where $\tau$ is a smooth irreducible representation of $D_\pp^{\times}$. 

If $\zeta_{\mathrm B}$ is 
locally algebraic then $\zeta_{\mathrm B}(x)= x^c \zeta_{\sm}(x)$ for all $x\in \Qp^\times$, where $\zeta_\sm$ is a
smooth character. After possibly  twisting $\mathrm B$ by  the character $x \mapsto \sqrt{ x|x|}$ 
(or its inverse) we may assume that $a-c$ is even. Then $\Hom_{U(\glt)}(\Sym^a L^2\otimes \det^b, \mathrm B)\neq 0$, where $a-c=2b$ and we may conclude as before.
\end{proof}

\begin{remar} The characters appearing in Proposition \ref{1-lalg} are not uniquely determined, since we may write the character $\det^a$ as a product of a unitary character 
$\pr(\det^a)$, and a smooth character $(\det^a \pr(\det^a)^{-1})$.
\end{remar}

We say that an irreducible component of a potentially semistable deformation ring is of \textit{discrete series type} if the closed points in the generic fibre corresponding to Galois representations, 
which do not become crystalline after restricting to the Galois group of an abelian extension, are Zariski dense in that component.

\begin{prop}\label{yeti} Assume the set up of Theorem \ref{main}. Let  $\Shat^1(\Pi)^{\onealg}$
be the subset of $D_{\pp}^{\times, 1}$-locally algebraic vectors in $\Shat^1(\Pi)$, where $D_{\pp}^{\times, 1}$
is the subgroup of $D_{\pp}^{\times}$ of elements with reduced norm equal to $1$.
Then $\Shat^1(\Pi)^{\onealg}$ is a finite dimensional $L$-vectors space. Moreover, 
if $\Shat^1(\Pi)^{\onealg}$ is non-zero then a twist of $\rho_x|_{G_{F_\pp}}$ by a character, 
defines a point lying on an irreducible component of discrete series type of some potentially semi-stable deformation ring of $\rbar$.
\end{prop}
\begin{proof} After twisting by a character  $\chi: (\AfF)^{\times}/F^{\times}\rightarrow 1+\pp$, 
which  is trivial on  $U^p \cap (\AfF)^{\times}$ we may assume that the restriction of $\psi$ to 
$(\AfF)^\times \cap U_p$ is locally algebraic.   

Let $\sigma:=\sigma_{\alg}\otimes\sigma_{\sm}$ be an irreducible locally algebraic representation of $U_{\pp}=\OO_{D_{\pp}}^{\times}$ with central character $\zeta$, where 
$\sigma_{\alg}=\Sym^b L^2\otimes \det^a$ and $\sigma_{\sm}$ is smooth. To $\sigma_{\sm}$ 
one may attach an inertial type $\tau: I_{\Qp}\rightarrow \GL_2(L)$, such that $\tau$ extends to 
a representation of the Weil group $W_{\Qp}$, the kernel of $\tau$ is an open subgroup of $I_{\Qp}$ 
and the following holds: if $\pi'$ is a smooth irreducible representation of $D_{\pp}^{\times}$, $\pi$ 
is a smooth irreducible representation of $\GL_2(\Qp)$ corresponding to $\pi'$ via the classical 
Jacquet--Langlands correspondence, and $\rec_p(\pi)$ is the Weil--Deligne representation corresponding
to $\pi$ via the classical local Langlands correspondence then 
$\Hom_{U_{\pp}}(\sigma_{\sm}, \pi')\neq 0$ if and only if $\rec_p(\pi)|_{I_{\Qp}}\cong \tau$, see 
\cite[Thm.\,3.3]{gee_geraghty}. Note that in that case either $\pi$ is supercuspidal and  $\tau$ extends 
to an irreducible representation of $W_{\Qp}$ and the monodromy operator of $\rec_p(\pi)$ is zero, or 
$\pi$ is special series and the monodromy operator of $\rec_p(\pi)$ is non-zero.  Moreover, the inertial type $\tau$ determines $\sigma_{\sm}$ up to conjugation by a uniformizer $\varpi_D$ of $D$.

Let $\wt=(1-a, -a-b)$ and let $R^{\square}_{\rbar}(\wt, \tau, \psi)$ be the framed deformation ring 
of $\rbar$ parameterising potentially semistable lifts of $\rbar$, which have Hodge--Tate weights
equal to $\wt$, inertial type $\tau$ and determinant $\psi\varepsilon^{-1}$. Let $R^{\square}_{\rbar}(\wt, \tau, \psi)^{\mathrm{ds}}$ be the closure of closed points in the generic fibre, which  do not correspond to crystabelline representations. In the supercuspidal case $R^{\square}_{\rbar}(\wt, \tau, \psi)^{\mathrm{ds}}$ and $R^{\square}_{\rbar}(\wt, \tau, \psi)$ coincide. In the special series case, 
$R^{\square}_{\rbar}(\wt, \tau, \psi)^{\mathrm{ds}}$ is the union of irreducible components of  $R^{\square}_{\rbar}(\wt, \tau, \psi)$ of discrete series type and the potentially crystalline locus has codimension $1$.
Mapping a deformation to its trace induces a natural map $R_{\tr \rbar}^{\ps}\rightarrow R^{\square}_{\rbar}(\wt, \tau, \psi)^{\mathrm{ds}}$ and we let $R_{\tr \rbar}^{\ps}(\wt, \tau, \psi)^{\mathrm{ds}}$ be its image.

 We claim that the action of $R^{\ps}_{\tr \rbar}$ on 
$\Hom_{U_\pp}(\sigma, \widehat{H}^1_{\psi}(U^p, \OO)_\mm\otimes_{\OO} L)$ via the homomorphism 
$R_{\tr \rbar}^\ps\rightarrow \TT(U^p)_{\mm}$ factors through the quotient $R_{\tr \rbar}^{\ps}(\wt, \tau, \psi)^{\mathrm{ds}}$. It follows from Proposition \ref{projective} that as $U_p^{\pp}$-representation 
$\Hom_{U_\pp}(\sigma, \widehat{H}^1_{\psi}(U^p, \OO)_\mm\otimes_{\OO} L)$ is isomorphic to 
a finite direct sum of copies of the Banach space $C_{\psi}(U^{\pp}_p)$ of continuous functions $f: U_p^{\pp}\rightarrow L$, 
on which the centre acts by $\psi$ and $U^{\pp}_p$ acts by right translations. Since the action of $R^{\ps}_{\tr \rbar}$ is continuous it is enough to 
show that it factors through $R_{\tr \rbar}^{\ps}(\wt, \tau, \psi)^{\mathrm{ds}}$ on a dense subspace. 
For this choose an open subgroup $V^{\pp}_p$ of $U^{\pp}_p$ such that the restriction of $\psi$ 
to the center of $V^{\pp}_p$ is algebraic, and let $\lambda$ be an algebraic representation of $V^{\pp}_p$ 
with central character $\psi$. The union of $\lambda$-isotypic subspaces inside $C_{\psi}(U^{\pp}_p)$
as $K_{p}^{\pp}$-representation taken over all open subgroups $K_p^{\pp}$ of $V^{\pp}_p$ will be dense, since locally constant functions are dense in the space of continuous functions. Thus 
it is enough to prove that the action of $R^{\ps}_{\tr \rbar}$ on 
$\Hom_{K_p^{\pp} U_{\pp}}(\lambda\otimes \sigma, \widehat{H}^1_{\psi}(U^p, \OO)_\mm\otimes_{\OO} L)$
factors through $R_{\tr \rbar}^{\ps}(\wt, \tau, \psi)^{\mathrm{ds}}$ for all open subgroups $K_p^\pp$.
It follows from Emerton's spectral sequence in \cite[Cor.2.2.8]{interpolate}, see \cite[Prop. 5.2]{newton_pJL}, 
that we have an isomorphism of $\TT(U^p)_{\mm}$-modules: 
$$ \Hom_{U_{\pp}}(\sigma_{\sm}, H^1(X(K_p^{\pp} U_{\pp}), \mathcal V_W)_{\mm}[\psi_{\sm}])\cong 
\Hom_{K_p^{\pp} U_{\pp}}(\lambda\otimes \sigma, \widehat{H}^1_{\psi}(U^p, \OO)_\mm\otimes_{\OO} L),$$
where $\mathcal V_W$ is a local system on $X(K_p^{\pp} U_{\pp})$ corresponding to the algebraic representation $W:= (\lambda\otimes \sigma)^*$ and $\psi_{\sm}=\psi\psi_{\alg}^{-1}$, and $\psi_{\alg}$ is the central character of $\lambda\otimes \sigma$. The left hand side of the equation is 
a semi-simple $\TT(U^p)_{\mm}[1/p]$-module and the eigenvalues correspond to an automorphic forms, see
\cite[\S 3]{interpolate}, satisfying local conditions imposed by $\sigma_{\sm}$ at $\pp$, $W$ at infinite places. The compatibility of local and global Langlands correspondence implies that if $\rho_x$ is the Galois representation attached to such automorphic form then $\tr \rho_x|_{G_{F_{\pp}}}$ gives a point
in $\Spec R_{\tr \rbar}^{\ps}(\wt, \tau, \psi)^{\mathrm{ds}}$. This finishes the proof of the claim.

Let us fix $x\in \mSpec \T(U^p)_\mm[1/p]$, let $y$ be its image in $\Spec R^{\ps}_{\tr \rbar}$ and assume that 
$$\Hom_{U(\slt)}(\Sym^b L^2, (\widehat{H}^1_{\psi, \lambda}(U^{\pp}, \OO)\otimes_{\OO} L)[\mm_x])\neq 0.$$
It follows from Proposition \ref{1-lalg} that 
$$\Hom_{U_{\pp}}(\sigma \otimes \eta\circ \Nrd, (\widehat{H}^1_{\psi, \lambda}(U^{\pp}, \OO)\otimes_{\OO} L)[\mm_x])\neq 0,$$
where $\sigma$ is as above and $\eta$ is either trivial or $\sqrt{\pr(\chi_{\cyc})}$. If $\eta$ is trivial 
then the claim implies that $y$ lies in $\Spec R^{\square}_{\rbar}(\wt, \tau, \psi)^{\mathrm{ds}}$.
We note that our assumption on $\rbar$ implies that any reducible potentially semistable lift $r$ with Hodge-Tate weights $\wt$ is crystabelline and its $\WD(r)|_{W_{\Qp}}$ is a direct sum of distinct characters. Such representations cannot correspond to points on irreducible components of discrete series type. Hence $\rho_x|_{G_{F_{\pp}}}$ is irreducible and thus is determined by its trace.

If $\eta$ is not trivial then by twisting by its inverse and using the claim again, we deduce that $\tr (\rho_x|_{G_{F_\pp}}\otimes \eta^{-1})$ gives a point in $\mSpec R^{\square}_{\rbar}(\wt, \tau, \psi \eta^{-2})^{\mathrm{ds}}[1/p]$. 
Note that the character $\psi\eta^{-2}=\psi \pr(\chi_{\cyc})^{-1}$ is locally algebraic, but the representation
$\lambda\otimes \eta^{-1}\circ\Nrd$ is not. That is why we cannot appeal directly to the results of Emerton in this case. Thus $\rho_x|_{G_{F_{\pp}}}$ determines the integers $a$ and $b$, the representation 
$\sigma_{\sm}$ up to its conjugate by the uniformizer $\varpi_D$ of $D_{\pp}^{\times}$, and whether $\eta$ is trivial or not, by comparing whether the Hodge--Tate weight of $\det\rho_x|_{G_{F_{\pp}}}$
has the same parity as the Hodge--Tate weight of $\psi|_{G_{F_{\pp}}}$.
 This implies that $\Shat^1(\Pi)^{\onealg}$ is isomorphic as a representation of $U_{\pp}$ to a finite direct 
sum of copies of $\sigma \otimes \eta\circ \Nrd$ or its conjugate by $\varpi_D$. In particular, $\Shat^1(\Pi)^{\onealg}$ is finite dimensional.
\end{proof}

\begin{remar} If $F=\QQ$ then the proof of Proposition \ref{yeti} can be simplified, 
since after twisting we may directly appeal to the results of Emerton on locally algebraic vectors in completed cohomology.
\end{remar}

\begin{thm}\label{main_2} Assume the set up of Theorem \ref{main} and the notation of Proposition \ref{yeti}. 
The quotient $\Shat^1(\Pi)/\Shat^1(\Pi)^{\onealg}$ contains an irreducible closed subrepresentation of $\delta$-dimension $1$.
\end{thm}
\begin{proof} Since the $\delta$-dimension of $\Shat^1(\Pi)$ is $1$ by Proposition \ref{delta_Pi} and 
the $\delta$-dimension of $\Shat^1(\Pi)^{\onealg}$ is $0$ by Proposition \ref{yeti} and Lemma \ref{dim0implies}, the quotient $\Shat^1(\Pi)/\Shat^1(\Pi)^{\onealg}$ is non-zero and has $\delta$-dimension $1$. Since it is admissible it will contain an irreducible subrepresentation $\mathrm B$. We have $\delta(\mathrm B)\le 1$. 
If $\delta(\mathrm B)=1$ then we are done, otherwise $\delta(\mathrm B)=0$ and so $\mathrm B$ is 
a finite dimensional $L$-vector space by Lemma \ref{dim0implies}. The extension 
of $\Shat^1(\Pi)^{\onealg}$ by $\mathrm B$ inside $\Shat^1(\Pi)$ is a finite dimensional 
$L$-vector space, and Proposition \ref{1-lalg} implies that the action of $D_{\pp}^{\times, 1}$ on it is locally algebraic. 
Since $\Shat^1(\Pi)^{\onealg}$ is the maximal subspace of $\Shat^1(\Pi)$ with this property we conclude that $\mathrm  B=0$.
\end{proof}

\end{document}